\documentclass[10pt]{amsart}
\usepackage{academicons}
\usepackage{hyperref}
\usepackage[bottom]{footmisc}
\usepackage{geometry}
\usepackage{graphicx}
\usepackage{enumitem}                               
\usepackage{amssymb}
\usepackage{amsmath}
\usepackage{amsthm}
\usepackage{xcolor}
\usepackage{ulem}
\newtheorem{theorem}{Theorem}

\newtheorem{corollary}[theorem]{Corollary}

\newtheorem{theoremn}{Theorem}[section] 
\newtheorem{lemman}[theoremn]{Lemma}
\newtheorem{propositionn}[theoremn]{Proposition}

\newtheorem{corollaryn}[theoremn]{Corollary}

\newtheorem*{theorem*}{Theorem}
\newtheorem*{lemma*}{Lemma}

\theoremstyle{definition}
\newtheorem{definition}[theoremn]{Definition}

\newtheorem{remark}[theoremn]{Remark}
\newtheorem{example}[theoremn]{Example}

\newtheoremstyle{definition*}
{\topsep}
{\topsep}
{}
{0pt}
{\bfseries}
{.}
{ }
{\thmname{#1}\thmnumber{ #2}\thmnote{ (#3)}}

\theoremstyle{definition*}
\newtheorem*{definition*}{Definition}
\newtheorem*{remark*}{Remark}
\newtheorem*{claim*}{Claim}

\newcommand{\E}{\mathcal E}
\newcommand{\D}{\mathcal D}

\newcommand{\F}{\mathcal F}

\begin{document}

	\baselineskip 13pt

\title[A generalization of  Alperin Fusion theorem ]{A generalization of  Alperin Fusion theorem \\ and \\ its applications}

\author{{M.Yas\.{I}r} K{\i}zmaz }
\address{ Martin-Luther-Universit\"at Halle-Wittenberg Institut für Mathematik, Theodor-Lieser-Str. 5, 06120 Halle, Germany}
\email{ysrkzmz@gmail.com}


\begin{abstract}
Let $\F$ be a saturated fusion system on a finite $p$-group $S$, and let $P$ be a strongly $\F$-closed subgroup of $S$. We define the concept ``$\F$-essential subgroups with respect to $P$" which are some proper subgroups of $P$ satisfying some technical conditions, and show that an $\F$-isomorphism between subgroups of $P$ can be factorised by some automorphisms of $P$ and $\F$-essential subgroups with respect to $P$. When $P$ is taken to be equal $S$,  Alperin-Goldschmidt fusion theorem can be obtained as a special case. We also show that $P\unlhd \F$ if and only if there is no $\F$-essential subgroup with respect to $P$. The following definition is made: a $p$-group $P$ is \textit{strongly resistant} in saturated fusion systems if $P\unlhd \F$ whenever there is an over $p$-group $S$ and a saturated fusion system $\F$ on $S$ such that $P$ is strongly $\F$-closed. It is shown that several classes of $p$-groups are strongly resistant, which appears as our third main theorem. We also give a new necessary and sufficient criteria for a strongly $\F$-closed subgroup to be normal in $\F$. These results are obtained as a consequences of developing a theory of quasi and semi-saturated fusion systems, which seems to be interesting for its own right. 
\end{abstract}
\maketitle
\section{Introduction}
 The study of finite groups has been significantly enriched by the development of powerful tools and techniques that have provided deep insights into their structure. Among these tools, the Alperin fusion theorem, more precisely the Alperin-Goldschmidt fusion theorem, stands out as a fundamental result that has played a crucial role in our understanding of finite groups. The theorem, which concerns the fusion of elements in $p$-subgroups of a given finite group, has found applications in diverse areas of mathematics including group theory and has inspired a vast body of research.

The concept of saturated fusion systems has emerged as an important generalization of finite groups, allowing for a more unified and elegant treatment of various group-theoretic phenomena. Saturated fusion systems can be regarded as abstract categories that retain essential information about the fusion of elements in $p$-subgroups, without necessarily requiring the full group structure. This abstraction has proven to be remarkably fruitful, paving the way for new advances in the study of finite groups and related algebraic structures.

In the present paper, we embark on an investigation of the Alperin fusion theorem in the context of saturated fusion systems. Our goal is to obtain a generalization of this fundamental theorem that is applicable to saturated fusion systems, thereby extending its reach and enhancing its utility in the study of these appealing structures. Through a careful examination of the key properties of saturated fusion systems, we develop a variant of the Alperin fusion theorem, which provides additional insights into a morphism situated within a strongly closed subgroup and elucidates how this strongly closed subgroup is embedded in a saturated fusion system. This enhanced understanding will contribute to a deeper exploration of strongly closed subgroups of fusion systems.

Our work builds upon and extends previous research in the area, providing a valuable contribution to the ongoing exploration of saturated fusion systems and their interplay with finite groups. We believe that our generalization of the Alperin fusion theorem will prove to be a useful tool for researchers working in this field, opening up new avenues for further study and fostering a deeper understanding of the rich and intricate world of finite group theory.

Before establishing the theorem, we first need to make some definitions. We should also note that all groups considered in this paper are finite. \newline

\textbf{Definitions}
\begin{enumerate}
	
	\item[(\textbf{a})] 	Let $G$ be a group and $S\in Syl_p(G)$, and $D>1$ be a strongly closed subgroup in $ S$ with respect to $G$. We say that  a proper subgroup $H$ of $G$ is \textbf{strongly $D$-embedded in $G$} if
	\begin{enumerate}
		\item[i)] $D^x \leq H$ for some $x\in G$.
		\item[ii)]   for all $g\in G-H$, the subgroup $H\cap H^g$ does not contain any $G$-conjugate of any nontrivial subgroup $U$ of $D$.
	\end{enumerate}

 \item[(\textbf{b})]   Let $S$ be a $p$-group and $\F$ be a saturated fusion system on $S$, and $P$ be a strongly $\F$-closed subgroup of $S$. A subgroup $Q$ of $P$ is called \textbf{ $\F$-centric with respect to $P$} if $C_P(Q\psi)\leq Q\psi$ for all $\psi \in Hom_{\F}(Q,P)$. Moreover, we say that a proper subgroup $Q$ of $P$ is an \textbf{$\F$-essential subgroup with respect to $P$} if \begin{enumerate}
 	\item[i)] $Q$ is an $\F$-centric subgroup with respect to $P$.
 	\item[ii)] $Q$ is fully normalized.
 	\item[iii)]  $Out_{\F}(Q)$ has a strongly $Out_{P}(Q)$-embedded subgroup. 
                       \end{enumerate} 

\end{enumerate}

We observe that when $D=S$, the condition for $H$ to be strongly $D$-embedded in $G$ is equivalent to being strongly $p$-embedded in $G$. With this observation, we can proceed to establish the generalization of the Alperin fusion theorem:

\begin{theorem} \label{main thm A}
	Let $\F$ be a saturated fusion system on a $p$-group  $S$ and $P$ be a strongly $\F$-closed subgroup of $S$, and  let $Q,R \leq P$ and  $\psi:Q\to R$ be an $\F$-isomorphism. Then there exists
	
	\begin{enumerate}
		\item[(a)] a sequence of isomorphic subgroups $Q = Q_0, Q_1, . . . , Q_{n} = R$ of $P$,
		
		\item[(b)] a sequence : $S_1, S_2, . . . , S_n$ of $P$ where each $S_i$ is an $\F$-essential subgroup  with respect to $P$  or $S_i=P$ such that $Q_{i-1},Q_i\leq S_i$ for $i=1,\ldots, n$,
		
		\item[(c)] $\psi_i\in Aut_{\F}(S_i)$ such that $( Q_{i-1})\psi_i=Q_i$ and $\psi=(\psi_1|_{Q_1})\circ(\psi_2|_{Q_2}) \circ \ldots \circ (\psi_n|_{Q_n})$.
	\end{enumerate}

\end{theorem}

The aforementioned theorem essentially states that an $\F$-isomorphism between subgroups of a strongly $\F$-closed subgroup $P$ can be factored through certain automorphisms of $P$ and $\F$-essential subgroups with respect to $P$.

Consider $Q$ as a fully normalized subgroup of $P$. We will later demonstrate that $Aut_P(Q)$ is strongly closed in $Aut_S(Q)$ with respect to $Aut_{\F}(Q)$ (see Lemma \ref{lem aut_p(Q) is strongly closed}). Consequently, $Out_P(Q)$ is strongly closed in $Out_S(Q)$ with respect to $Out_{\F}(Q)$. To ensure that $Q$ is essential with respect to $P$, we impose the additional condition that $Q$ is an $\F$-centric subgroup with respect to $P$, and $Out_{\F}(Q)$ contains a strongly $Out_P(Q)$-embedded subgroup. This results in certain constraints on the structure of $Aut_{\F}(Q)$ and $Out_{\F}(Q)$ (see Proposition \ref{prop charac. of essentails}):

\begin{enumerate}
	
	\item[(a)]  $Out_{P}(Q)\cap O_p(Out_{\F}(Q))=1.$
	\vspace{.2cm}
	\item[(b)] $Inn(Q)=Core_G(D)$ where $G=Aut_{\F}(Q)$ and $D=Aut_P(Q)$.
\end{enumerate}

 Therefore, our essential subgroups with respect to $P$ deviate somewhat from the ``typical ones" and, in fact, possess a more general form:  If we take $P=S$, that $Out_{\F}(Q)$ has a strongly $Out_{P}(Q)$-embedded subgroup is equivalent to say that $Out_{\F}(Q)$ has a strongly $p$-embedded subgroup as $Out_{P}(Q)\in Syl_p(Out_{\F}(Q))$ in that case, and so $Q$ turns to be an ``typical essential" subgroup. Consequently, the Alperin-Goldschmidt fusion theorem corresponds to the special layer where $P=S$ in our theorem.

\begin{theorem} \label{A' theroem}
	Let $\F$ be a saturated fusion system on a $p$-group  $S$ and $P$ be a strongly $\F$-closed subgroup of $S$. Then $P\unlhd \F$ if and only if there is no $\F$-essential subgroup with respect to $P$.

\end{theorem}

The above theorem supplies a natural machinery to investigate whether $P$ is normal in $\F$ by relying only on the internal structure of $P$. In the following theorem, we shall observe that when $P$ is in some certain  class of $p$-groups, it is always normal in  $\F$, regardless of the structure of over group $S$ and saturated fusion system $\F$ on $S$. For the next theorem, recall that a $p$-group $P$ is said to be \textbf{generalized extraspecial} $p$-group if $\Phi(P)=P'\cong C_p.$
	
	\begin{theorem}\label{main thm B}
			Let $\F$ be a saturated fusion system on a $p$-group  $S$ and $P$ be a strongly $\F$-closed subgroup of $S$. Suppose that  one of the following hold: 
			\begin{enumerate}
				\item[(a)] $P$ is a generalized extraspecial $p$-group such that $P$ is not in the form of  $E\times A$ where $A$ is an elementary abelian group and $E$ is
				
				\begin{enumerate}
					\item[(i)] a dihedral group of order $8$ when $p=2$.
					
					\item[(ii)] an extraspecial $p$-group of order $p^3$ with exponent $p$ when $p>2$.
					
				\end{enumerate} 
			
				\item[(b)] $p>2$ and $P$ is a metacyclic $p$-group.
				\item[(c)] $p=2$ and $P$ is a metacyclic $p$-group such that $P$ is not dihedral, semidihedral or generalized quaternion.
				
				\item[(d)] $P$ is isomorphic to one of the following $p$-groups of rank $2$:
				\begin{enumerate}
				 \item[(i)] $C(p,r)=\langle a^p=b^p=c^{p^{r-2}}=1 \mid [a,b]=c^{p^{r-3}}, [a,c]=[b,c]=1 \rangle$ where $p\geq 3$  and $r\geq 4$.
				
				\item[(ii)] $ G(p,r,\epsilon)=\langle a^p=b^p=c^{p^{r-2}}=1 \mid [b,c]=1, [a,b^{-1}]=c^{p^{\epsilon r-3}}, [a,c]=b \rangle$ where $p\geq 5$ and $r\geq 4.$
			\end{enumerate}
			\end{enumerate}

		Then $P\unlhd \F$, that is, $\F=N_{\F}(P)$. Moreover, if $(c)$ holds and $P$ is not homocyclic abelian, then $\F=SC_{\F}(P)$.
	\end{theorem}

The proof of Theorem C is considered the most challenging and technical within this article, leading to the dedication of an entire chapter, Chapter IV (Applications), to its demonstration.  In this section, we also develop some new techniques related to the coprime action and its use in fusion systems as an important tool. We now leave this discussion to Chapter IV and focus on the consequences of this theorem.
	
Recall that a $p$-group $S$ is called \textbf{resistant in saturated fusion systems} if $S\unlhd \F$  for any saturated fusion system $\F$ on $S$. 
	The above theorem leads us to make the following definition:
	
	We say that a $p$-group $P$ is \textbf{strongly resistant in saturated fusion systems} if $P\unlhd \F$ whenever there is an over $p$-group $S$ and a saturated fusion system $\F$ on $S$ such that $P$ is strongly $\F$-closed. Only known strongly resistant $p$-groups are abelian $p$-groups, which is due to G. Glauberman, and hence the above theorem demonstrates the first examples of nonabelian class of $p$-groups which are strongly resistant. The $(b)$ and $(d)$ part of the theorem also implies that if $P$ is of rank $2$ and is not of maximal class then $P$ is strongly resistant in saturated fusion systems when $p\geq 5$ (see  \cite[ Theorem A.1]{Drv} for the list of $p$-groups of rank 2). 
	
	 We would like to note that the class of $p$-groups in the above theorem are already known to be resistant in saturated fusion systems  (see \cite[Theorem 5.3, Proposition 5.4]{Stan}),  \cite[Theorem 3.7]{Crv2} and \cite[Theorems 4.2 and 4.3]{Drv}). Hence, Theorem C is a vast generalization of those results. 

We adopt a useful definition from \cite{myk}. Let $\Gamma$ be  the nontrivial semidirect product of $\mathbb Z_4$ by $\mathbb Z_4$.
We say that $P$ is of \textbf{odd type} if  $P$ has no subgroup isomorphic to $Q_8$ and $\Gamma$. Notice that $P$ is always of odd type when $p$ is odd. Recall that $\Omega_i(P)$ denote the subgroup $\langle \{x\in P\mid x^{p^i}=1 \}\rangle$ of $P$ for $i=1,2,...$ and we simply use $\Omega(P)$ in place of $\Omega_1(P)$. Now set \[\Omega^*(P)=\begin{cases} 
	\Omega(P) & \text{if $P$ is off odd type} \\
	\Omega_2(P) & otherwise
	
\end{cases}
\]
The following theorem generalizes a result of Aschbacher (see \cite[Proposition 4.62]{Crv}).

\begin{theorem}\label{main thm C}
Let $\F$ be a saturated fusion system on a $p$-group $S$ and $P$ be a strongly $\F$-closed subgroup of $S$. Then $P\unlhd \F$ if and only if $\Omega^*(P)$ has a subgroup series $1=Q_0\leq Q_1\leq Q_2\ldots \leq Q_{n-1}\leq Q_n=\Omega^*(P)$ such that 

\begin{enumerate}
\item[(a)] $Q_i$ is strongly $\F$-closed for $i=0,1,\ldots, n$.
\item[(b)] $[Q_{i+1},P]\leq Q_i$ for $i=0,1,...,n-1.$
\end{enumerate}

\end{theorem}

As an application of the above theorem, we obtain another class of $p$-groups which are strongly resistant.
 
\begin{corollary}
	Let $P$ be a $p$-group of odd type such that $\Omega(P)\leq Z(P)$. Then $P$ is strongly resistant in saturated fusion systems.
	\end{corollary}

We believe that Theorem \ref{main thm A} holds significant potential for generalizing numerous theorems whose proofs are based on the Alperin fusion theorem. As such, it may have applications that extend beyond those discussed in the current work. An interesting corollary of Theorem A, which merits separate mention, is the generalization of the Frobenius $p$-nilpotency theorem in finite groups. This corollary exemplifies the broader implications and reach of Theorem \ref{main thm A}, highlighting its potential for further developments in the study of group theory.

\begin{corollary}[A general version of Frobenius $p$-nilpotency theorem]
Let $G$ be a group and $S\in Syl_p(G)$, and $D$ be a strongly closed subgroup in $ S$ with respect to $G$.  Then $G$ is $p$-nilpotent if and only if $N_G(U)$ is $p$-nilpotent for all $U\leq D$ such that $C_D(U)\leq U$.
\end{corollary}

\textbf{The organization of the paper:} In the subsequent section, we will introduce the notation and conventions employed in this work, along with some preliminary information about saturated fusion systems and coprime actions. This foundation will provide the necessary background for our study.

In Section 3, we will explore quasi and semi-saturated fusion systems, proving analogues of the Alperin fusion theorem for these classes of fusion systems (see Theorems \ref{thm Alperin 1} and \ref{thm Alperin 2}). We will obtain characterizations of essential subgroups in semisaturated fusion systems (see Proposition \ref{prop charac. of essentails} and Corollary \ref{cor equivalent def of essential}). Additionally, a natural correspondence between semisaturated and saturated fusion systems will be established, which will lead to the proofs of Theorems A, B, and D.

In Section 4, we will rely on the following equivalence:  \textit{A $p$-group $P$ is strongly resistant in saturated fusion systems if and only if $P$ is resistant in semisaturated fusion systems} (see Corollary \ref{cor correspondence}).

 Each class of $p$-groups appearing in Theorem C will be examined separately to demonstrate their resistance in semisaturated fusion systems. Subsequently, Theorem C will be established as a result of Corollary \ref{cor correspondence}. This structured approach will allow us to thoroughly investigate the properties and relationships between these classes of $p$-groups and fusion systems. 
\section{Preliminaries}

\subsection{Some basics related to fusions systems}

Let $\F$ be a fusion system on a $p$-group $P$ and $Q\leq P$. 
\begin{enumerate}

\item[$\bullet$] A subgroup $R$ of $P$ is called an \textbf{$\F$-conjugate} of $Q$, if there exists  $\phi\in Hom_{\F}(Q,R)$	such that $R=Q\phi $. The set of all $\F$-conjugates of $Q$ is \textbf{denoted} by $Q^{\F}$.

\item[$\bullet$] For $x\in P$, we set $x^{\F}=\{ x\phi \mid \phi \in Hom_{\F}(\langle x \rangle,P)  \}$. 

\item[$\bullet$] $Q$ is said to be \textbf{strongly $\F$-closed} if $x^{\F}\subseteq Q$ for all $x\in Q$.

\item[$\bullet$] We say that $Q$ is \textbf{fully $\F$-normalized} if $|N_P(Q)|\geq |N_P(R)|$ for every $\F$-conjugate $R$ of $Q$. Similarly, $Q$ is said to be  \textbf{fully $\F$-centralized}  if $|C_P(Q)|\geq |C_P(R)|$ for every $\F$-conjugate $R$ of $Q$. We can drop ``$\F$" if it is obvious from the context.

\item[$\bullet$] For $x\in S$, the conjugation map induced by $x$ is denoted by $c_x$.

\item[$\bullet$] Let  $\phi\in Hom_{\F}(Q,P)$. We define 
$$	N_{\phi}(Q)=\{x\in N_P(Q) \mid \exists y\in N_P(Q\phi) \textit{ such that } \phi^{-1}(c_x)\phi =c_y  \text{ on } Q\phi\}.$$ We also note that the standard notation for $N_{\phi}(Q)$ is $N_{\phi}$.

\item[$\bullet$] We say that $Q$ is \textbf{receptive} if $\phi:R\to Q$ is an $\F$-isomorphism, then $\phi$ extends to $\overline \phi \in Hom_{\F}(N_{\phi}(R),P)$.

\item[$\bullet$] We say that $Q$ is \textbf{$\F$-centric}, if $C_P(R)\leq R$ for all $R\in Q^{\F}$.

\item[$\bullet$] Let $G$ be a group and $D\leq S \leq G$. We say that $D$ is \textbf{strongly closed} in $S$ with respect to $G$ if $x^g\in D$ whenever $x^g \in S$ for some $x\in D$ and $g\in G$.  Note that if $S\in Syl_p(G)$, $D$ is strongly closed in $S$ with respect to $G$ if and only if $D$ is strongly $\F$-closed in $S$ where $\F=\F_S(G)$, which is the fusion system induced by $G$ on $S$.

\end{enumerate}

\begin{definition}
Let $\F$ be a fusion system on a $p$-group $P$. We say $\F$ is \textbf{saturated} provided that \begin{enumerate}[label=(\alph*)]
	\item $Aut_P(P)$ is a Sylow $p$-subgroup of $Aut_{\F}(P)$, and
	\item For every subgroup $Q$ of $P$ and every morphism $\phi\in Hom_{\F}(Q,P) $ such that $Q\phi$ is fully normalized extends to a morphism $ \overline \phi\in Hom_{\F}(N_{\phi}(Q),P)$.
\end{enumerate}	
\end{definition}

Let $Q\leq P$ be fully $\F$-normalized where $\F$ is saturated. Then it is well known that $Aut_{P}(Q)\in Syl_p(Aut_{\F}(Q))$.

 The following type of calculations are also standard and used frequently.
		
		
\begin{lemman}\label{iden2}
	Let $Q\leq P$, $s\in N_P(Q)$, $\phi\in Hom_{\F}(Q,P)$ and $\overline \phi\in Hom_{\F}( \langle Q,s\rangle,P)$ such that $\overline\phi$ is an extension of $\phi$. Then $\phi^{-1}c_s\phi=c_{s \overline\phi}$ on $Q\phi$.
\end{lemman}

\subsection{Coprime actions}	

Let $A$ be a group acting on a group $G$ via automorphisms. We say that $A$ acts on $G$ \textbf{coprimely} if $gcd(|A|,|G|)=1$.

\begin{lemman}\label{coprime action}
Let $A$ be a group acting on a group $G$ coprimely. Then the following hold:
\begin{enumerate}
	\item[(a)] $G=[G,A]C_G(A)$. Moreover, if $G$ is abelian, then $G=[G,A]\times C_G(A)$.
	\item[(b)]  $[G,A,A]=[G,A]$.
	\item[(c)] Let $N$ be an $A$-invariant normal subgroup of $G$ and set $\overline G=G/N$. Then $C_{\overline G}(A)=\overline{C_G(A)}.$
	\item[(d)] Write $\overline G=G/\Phi(G)$. If $[\overline G, A]=1$, then $[G,A]=1$.
\end{enumerate}
\end{lemman}

\begin{proof}[\textbf{Proof.}]
See \cite[Corollaries 3.28 and 3.29]{Isc} for (c) and (d), and see \cite[Lemmas 4.28 and 4.29, Theorem 4.34]{Isc} for (a) and (b).
\end{proof}

Recall that a $p$-group $P$ is of odd type if  $P$ has no subgroup isomorphic to $Q_8$ and $\Gamma$ ($\Gamma$ denote group isomorphic to the nontrivial semidirect product of $\mathbb Z_4$ by $\mathbb Z_4$). In the case that $p$ is odd, $P$ is always of odd type. Set \[\Omega^*(P)=\begin{cases} 
	\Omega(P) & \text{if $P$ is off odd type} \\
	\Omega_2(P) & otherwise
	
\end{cases}
\]
\begin{lemman}\label{coprime action 2}
	Let $A$ be a group acting on a $p$-group $P$ coprimely. Then the following hold:
	\begin{enumerate}
		\item[(a)] If $P$ is cyclic and $C_P(A)\neq 1$, then $[P,A]=1$, that is, $C_P(A)=P$.
		\item[(b)] If $[\Omega^*(P),A]=1$, then $[P,A]=1$.
	
	\end{enumerate}
\end{lemman}
\begin{proof}[\textbf{Proof.}]
(a) follows from Lemma \ref{coprime action}(a), and (b) follows from  \cite[Corollary D]{myk}.
\end{proof}

\begin{lemman}\label{coprime action 3}
Let $P$ be a $p$-group and $A\leq Aut(P)$. Then the following hold:

\begin{enumerate}
	\item[(a)] $C_A(\Omega^*(P))\leq O_p(A)$ and $C_A(P/\Phi(P))\leq O_p(A)$.
	\item[(b)] Suppose that $A$ stabilises a subgroup series $1=S_0\leq S_1 \leq S_2\leq ... \leq S_{n-1}\leq S_n=\Omega^*(P)$. If $B\leq A$ such that $[S_{i+1},B]\leq S_i$ for $i=0,1,..,n-1$, then $B\leq O_p(A)$.
	
\end{enumerate}

\end{lemman}
\begin{proof}[\textbf{Proof.}]
	$(a)$ directly follows from Lemma \ref{coprime action 2}(b) and Lemma \ref{coprime action}(d).
	
	$(b)$ Since each $S_i$ is $A$-invariant, $N=\bigcap_{i=0}^{n-1} C_A(S_{i+1}/S_i)$ is a normal subgroup of $A$ that contains $B$. Let $a \in N$ of order coprime to $p$. Since $[\Omega^*(P),a,...,a]=1$, we obtain that $[\Omega^*(P),a]=1$ by Lemma \ref{coprime action}(b). It follows that $a$ is $p$-element by part (a), and so $a=1$. Hence, we obtain that $B\leq N \leq O_p(A)$ as desired.
\end{proof}

\section{Main Results}

In the first subsection, we focus on the foundational aspects of quasi and semi-saturated fusion systems. We define these concepts and establish their fundamental properties, which will be essential for our later discussions. While our primary focus will be on semisaturated fusion systems, we introduce the concept of quasi-saturated fusion systems as a means to develop a more general understanding of fusion systems. This will allow us to better comprehend the relationship between semisaturated and saturated fusion systems, and to explore the potential for further generalizations in the future.

It is important to note that, while our treatment of quasi-saturated fusion systems is limited in this article, we believe that the concept is valuable for the development of a more comprehensive theory. A more refined definition of semi-saturation could lie somewhere between the current definitions of quasi and semisaturated fusion systems, but our current focus on semi-saturation is well-suited for exploring its applications in saturated fusion systems. The use of quasi-saturated fusion systems in our proofs not only simplifies the arguments but also helps to break them down into more manageable components.

In the second subsection, we shift our attention to the proofs of Theorems A, B, and D, which involve saturated fusion systems as mentioned in the introduction. By establishing a natural correspondence between semisaturated and saturated fusion systems, we are able to demonstrate the broader applicability and versatility of our results. As we explore the essential properties of these systems, we will also showcase how the theorems proven in this article are more general than what is stated in the introduction.

After proving the Alperin fusion theorem for quasi-saturated fusion systems, we demonstrate that semisaturated fusion systems are indeed a subset of quasi-saturated fusion systems. This enables us to examine the properties of essential subgroups of semisaturated fusion systems more effectively. From this point onward, our focus will be on semisaturated and saturated fusion systems, leaving the exploration of potential examples of quasi-saturated fusion systems that are not semisaturated for future research.

In the final subsection, we will provide a variety of instructive examples, one of which demonstrates the existence of semisaturated fusion systems that are not saturated.

In summary, this section aims to establish a solid foundation for the study of semisaturated fusion systems and its natural correspondence with saturated one while also highlighting the potential for more general results involving quasi-saturated fusion systems.

\subsection{Quasi and Semi-saturated fusion systems}

\begin{definition}Let $\F$ be a fusion system on a $p$-group $P$ and $R< P$. We say that $R$ is \textbf{reproductive} and write $R\in \F^r$ if
	
	\begin{enumerate}
		\item[(a)] $R$ is receptive.
		\item [(b)] For each $Q\in R^{\F}$, there exists $\psi \in Hom_{\F}(Q,R)$ such that $N_{\psi}(Q)>Q$.
		
	\end{enumerate}
\end{definition}

\begin{definition} \label{def quasisaturated}Let $\F$ be a fusion system on a $p$-group $P$. We say that $\F$ is \textbf{quasisaturated} if each proper subgroup of $P$ has an $\F$-conjugate lying in $ \F^r.$  
\end{definition}

\begin{definition}\label{def essential subgroup}
Let $\F$ be a fusion system on a $p$-group $P$. A subgroup $Q$ of $P$ is called \textbf{$\F$-essential} if the following hold:

\begin{enumerate}
\item[a)] $Q\in \F^r$, that is, $Q$ is reproductive,
\item[b)] $H_Q$ is proper in $Aut_{\F}(Q)$ where $H_Q=\langle \phi\in Aut_{\F}(Q) \mid N_{\phi}(Q)>Q \rangle$.
\end{enumerate}

\end{definition}
Note that if the fusion system $\F$ is obvious from the context, we may simply say that $Q$ is an essential subgroup of $P$.

Let $\mathfrak V$ be a collection of subgroups of $P$. A subset $\mathfrak U $ of $ \mathfrak V$ is called an $\F$-\textbf{conjugacy representative} of $\mathfrak V$ if each member of $\mathfrak V$ has exactly one $\F$-conjugate in $\mathfrak U$.

\begin{definition}\label{def essential}
	Let $\F$ be a fusion system on a $p$-group $P$. A subset $\mathfrak{E}$ of $\F^r$ is called an \textbf{essential collection} if there exists an $\F$-conjugacy representative $\mathfrak{U}$ of $\F^r$ such that
	$$\mathfrak{E}=\{Q\in \mathfrak{U}\mid Q \textit{ is $\F$-essential}  \}.$$
\end{definition}

\begin{theoremn}\label{thm Alperin 1}
	Let $\F$ be a quasisaturated fusion system on a $p$-group  $P$ and $\mathfrak E$ be an essential collection. Let $Q,R \leq P$ and  $\psi:Q\to R$ be an $\F$-isomorphism. Then there exists
	
	\begin{enumerate}
		\item[(a)] a sequence of isomorphic subgroups $Q = Q_0, Q_1, . . . , Q_{n} = R$ of $P$,
		
		\item[(b)] a sequence of subgroups $S_1, S_2, . . . , S_n$ of $P$ where $S_i\in \mathfrak E $ or $S_i=P$ such that $Q_{i-1},Q_i\leq S_i$ for $i=1,\ldots, n$,
		
		\item[(c)] $\psi_i\in Aut_{\F}(S_i)$ such that $( Q_{i-1})\psi_i=Q_i$ and $\psi=(\psi_1|_{Q_1})\circ(\psi_2|_{Q_2}) \circ \ldots \circ (\psi_n|_{Q_n})$.
	\end{enumerate}
\end{theoremn}	
	\begin{proof}[\textbf{Proof.}]
		We proceed by induction on $|P:Q|$. When $Q=P$, the claim is obviously correct, and so assume $Q<P$. There exists an $\F$-conjugacy representative  $\mathfrak{U}$  of $\F^r$ such that $\mathfrak{E}=\{Q\in \mathfrak{U}\mid Q  \textit{ is essential}   \}$ by Definition \ref{def essential}.
		
		Let's first consider the case that $R\in  \mathfrak{U}$. Then there exists $\phi \in Aut_{\F}(R)$ such that $N_{\varphi}(Q)>Q$ where $\varphi=\psi\phi$ as $R$ is reproductive. We also see that $\varphi$ extends to $\overline \varphi\in Hom_{\F}(N_{\varphi}(Q),P)$ as $R$ is receptive. By induction applied to $N_{\varphi}(Q)$, we obtain that $\overline \varphi$ has such a decomposition, and so does $\varphi$. It follows that $\psi$ has such a decomposition if $\phi$ does. If $R$ is essential, then $R\in \mathfrak{E}$, and so we are done. Then we may assume that $R$ is not essential, that is, $H_R=Aut_{\F}(R)$ where $H_R$ is defined as in Definition \ref{def essential}. Thus, $Aut_{\F}(R)$ is generated by automorphisms $\phi_i$ for $i\in I$ such that $N_{\phi_i}(R)>R$. Moreover, $\phi_i$ extends to $ \overline \phi_i \in Hom_{\F}(N_{\phi_i}(R),P)$ as $R$ is receptive.  The inductive argument yields that $\overline \phi_i$ for $i\in I$ have such decompositions, and hence each $\phi_i$ has such a decomposition. It follows that $\phi$ has the desired decomposition as well, which completes the proof for this part.
		
		Now if $R\notin \mathfrak{U}$, then there exists $R^*\in R^{\F}$ such that $R^*\in \mathfrak{U}$ as $\F$ is quasisaturated and $\mathfrak{U}$ is an $\F$-conjugacy representative of $\F^r$.  Let $f\in Hom_{\F}(R,R^*).$ It follows that both $\psi f$ and $f$ have such decompositions by the previous paragraph, and so $\psi f\circ f^{-1}=\psi$ has such a decomposition as well.
	\end{proof}

\begin{lemman}\label{lem ess imp f-centric}
	Let $\F$ be a quasisaturated fusion system on a $p$-group $P$. Then each $\F$-essential subgroup of $P$ is $\F$-centric.
\end{lemman}

\begin{proof}[\textbf{Proof.}]
Let $Q$ be an essential subgroup of $P$. Then $H_Q$ is proper in $Aut_{\F}(Q)$, and so there exists $\phi \in Aut_{\F}(Q)$ such that $N_{\phi}(Q)=Q$. Since $C_P(Q)\leq N_{\phi}(Q)$, we get that $C_P(Q)\leq Q$. Moreover, $Q\in \F^r$, and so $Q$ is receptive, which yields that $Q$ is fully centralized (see \cite[Proposition 4.15]{Crv}). It then follows that $Q$ is $\F$-centric as desired. 
\end{proof}
 
\begin{definition}\label{def semisaturated}
	Let $P$ be a $p$-group and $\E$ be a fusion system on $P$. We say that $\E$ is \textbf{semisaturated} if there exists a over group $S\geq P$ and over fusion system $\F\geq \E$ such that
	
	\begin{enumerate}
		\item[(a)] $\F$ is saturated.
		\item[(b)]$P$ is strongly $\F$-closed.
		\item[(c)] $\E=\F_{\mid \leq P}$.
	\end{enumerate}
\end{definition}
In the above case, we say that $\E$ is a semisaturated fusion system induced by $(\F,S).$
Let $R\leq P$. We say that $R$ is fully $\F$-normalized if $|N_S(R)|\geq |N_S(Q)|$ for all $Q\in R^{\F}$, and $R$ is fully $\E$-normalized if $|N_P(R)|\geq |N_P(Q)|$ for all $Q\in R^{\E}$. Similarly, we shall use several wellknown definitions with ``with respect to $\E$ or $\F$" to make a distinction about whether the considered structure is $P$ with $\E$ or $S$ with $\F$. However, if the considered structure is obvious from the context, we may remove these adds. Note also that saturated fusion systems are also semisaturated as $\F=\F_{\mid \leq S}$.

\begin{lemman}\label{lem fully normalized w.r.t F}
	Let $\E$ be a semisaturated fusion system on a $p$-group $P$ induced by $(\F,S).$ Let $R\leq P$ and let $R$ be a fully $\F$-normalized. Then the following hold:
	
	\begin{enumerate}
		
		\item[(a)]  $R$ is a receptive subgroup of $P$ with respect to $\E$.
		\item[(b)] For each $Q\in R^{\E}$, there exists $\phi\in Hom_{\E}(Q,R)$ such that $N_{\phi}(Q)=N_P(Q)$ with respect to $\E$.
		
		\item[(c)] $R$ is fully $\E$-normalized.
	
	\end{enumerate}
	
	\begin{proof}[\textbf{Proof.}]
		
		$(a)$ Since $R$ is fully normalized  with respect to $\F$, it is also receptive with respect to $\F$. It then follows that $R$ is also receptive with respect to $\E.$
		
		$(b)$ Let $Q\in R^{\E}.$ Since $R$ is fully $\F$-normalized, there exists $\phi\in Hom_{\F}(Q,R)$ such that $N_{\phi}(Q)=N_S(Q)$ with respect to $\F$. On the other hand, $Hom_{\F}(Q,R)=Hom_{\E}(Q,R)$ as $Q,R\leq P.$ It then follows that $N_{\phi}(Q)=P\cap N_S(Q)=N_P(Q)$ with respect to $\E$.
		
		$(c)$ Let $Q\in R^{\E}$ such that $Q$ is fully $\E$-normalized. By part $(b)$, we can pick $\phi\in Hom_{\E}(Q,R)$ such that  $N_{\phi}(Q)=N_P(Q)$. Since $R$ is receptive by part $(a)$, $\phi$ extends to $\overline \phi \in Hom_{\E}(N_P(Q),P)$. Then we get that $\overline \phi (N_P(Q))\leq N_P(R)$.
		Since $Q$ is fully $\E$-normalized, we get that $R$ is also fully $\E$-normalized, which also yields that $\overline \phi (N_P(Q))= N_P(R).$
		
	\end{proof}
	
\end{lemman}

\begin{corollaryn}\label{cor phi extend Hom(N_P(Q), N_P(R)) }
	Let $\E$ be a semisaturated fusion system on a $p$-group $P$ induced by $(\F,S).$ Let $Q\leq P$ and $R\in Q^{\E}$. Assume that $Q$ is fully $\E$-normalized and $R$ is a fully $\F$-normalized. Then there exists $\phi \in Hom_{\E}(Q,R)$ such that $\phi$ extends to an isomorphism $\overline \phi \in Hom_{\E}(N_P(Q),N_P(R))$.
\end{corollaryn}

\begin{proof}[\textbf{Proof.}]
	See the proof of Lemma \ref{lem fully normalized w.r.t F}(c).
	\end{proof}

\begin{corollaryn}\label{cor semi imp quasisaturated}
	All semisaturated fusion systems are quasisaturated.
\end{corollaryn}
\begin{proof}[\textbf{Proof.}]
	
	Let $\E$ be a semisaturated fusion system on a $p$-group $P$ induced by $(\F,S)$ and let $Q<P$. Then $Q^{\F}=Q^{\E}$ as $P$ is strongly $\F$-closed. Now choose $R\in Q^{\E}$ such that $R$ is fully $\F$-normalized. Then we obtain that $R$ is reproductive by the Lemma \ref{lem fully normalized w.r.t F} (a) and (b), which proves the claim.
\end{proof}

\begin{lemman} \label{lem fully normalized w.r.t E}

	Let $\E$ be a semisaturated fusion system on a $p$-group $P$ and let $R$  a fully $\E$-normalized subgroup of $P$. Then the following hold:
	\begin{enumerate}
		
		\item[(a)]  $R$ is a receptive subgroup of $P$ with respect to $\E$.
		\item[(b)] For each $Q\in R^{\E}$, there exists $\phi\in Hom_{\E}(Q,R)$ such that $N_{\phi}(Q)=N_P(Q)$ with respect to $\E$.
		
	\end{enumerate}
\end{lemman}
\begin{proof}[\textbf{Proof.}]
 Let $\E$ be induced by $(\F,S)$ where $\F$ is saturated, and let $R^*\in R^{\E}=R^{\F}$ such that $R^*$ is fully $\F$-normalized.
	
	We can pick $f \in Hom_{\E}(R,R^*)$ such that $f$ extends to an isomorphism $$\overline f \in Hom_{\E}(N_P(R),N_P(R^*))$$ by Corollary \ref{cor phi extend Hom(N_P(Q), N_P(R)) }.
	
	 Let $Q\in R^{\E}$ and $\phi\in Hom_{\E}(Q,R)$. Set  $\phi \circ f=\psi$. Since $R^*$ is receptive by Lemma \ref{lem fully normalized w.r.t F} (a), $\psi$ extends to $\overline \psi \in Hom_{\E}(N_{\psi}(Q),N_P(R^*))$. We now claim that $N_{\phi}(Q)= N_{\psi}(Q)$. Let $x\in N_{\phi}(Q)$. Then there exists $y\in N_P(R)$ such that 
	${c_x}^{\phi}=c_y$. Then we have $${c_x}^{\psi}={c_x}^{\phi f}={c_y}^f=c_{y \overline f}.$$
	Note that we apply Lemma \ref{iden2} at the last equality.
	We get that $x\in N_{\psi}(Q)$  as $y\overline f \in N_P(R^*)$. Now let $x\in N_{\psi}(Q)$. Then there exists $z\in N_P(R^*)$ such that  ${c_x}^{\psi}={c_x}^{\phi f}=c_z$. It follows that $${c_x}^{\phi}={c_z}^{f^{-1}}={c_{z \overline f^{-1}}}$$
	
	by appealing Lemma \ref{iden2}. We get that $x\in N_{\phi}(Q)$ as $z \overline f^{-1}\in N_p(R)$. Consequently, $N_{\phi}(Q)=N_{\psi}(Q)$.
	Now we are ready to prove (a).
	 Observe that $\overline \psi \circ {\overline f}^{-1} \in Hom_{\E}(N_{\phi}(Q),N_P(R))$, and restriction of $\overline \psi \circ {\overline f}^{-1}$ to $Q$ is $\phi \circ f \circ f^{-1}=\phi$ as desired.
	 
	 Since $\psi=\phi \circ f$, we can choose $\phi$ such that $N_{\psi}(Q)=N_P(Q)$  by Lemma \ref{lem fully normalized w.r.t F} (b). Then $(b)$ follows from the fact that $N_{\phi}(Q)=N_{\psi}(Q)$.
\end{proof}

\begin{corollaryn} \label{cor fully norm imp reproductive}
	Let $\E$ be a semisaturated fusion system on a $p$-group $P$. Then proper  fully normalized subgroups of $P$ are reproductive. 
\end{corollaryn}

\begin{definition}\label{def main  essential}
	Let $\E$ be a semisaturated fusion system on a $p$-group $P$. An essential collection $\mathfrak{E}$  is called \textbf{a main essential collection} if each member of  $\mathfrak{E}$  is fully normalized.
\end{definition}

\begin{theoremn}\label{thm Alperin 2}
	Let $\E$ be a semisaturated fusion system on a $p$-group $P$. Then  a  main essential collection always exists (possibly empty). Moreover, for any main essential collection $\mathfrak E$, we have $\E=\langle Aut_{\E}(Q) \mid  Q\in \mathfrak E \ or \ Q=P  \rangle.$
\end{theoremn}

\begin{proof}[\textbf{Proof.}]
	Let $\mathfrak U$ be $\E$-conjugacy representative of proper fully normalized subgroups of $P$. By  Corollary \ref{cor fully norm imp reproductive}, we have $\mathfrak U \subseteq \E^r$. Moreover, $\mathfrak U$ is also $\E$-conjugacy representative of $\E^r$ as each proper subgroup has an $\E$-conjugate lying in $\mathfrak U$. Now set $\mathfrak{E}=\{Q\in \mathfrak{U}\mid Q \textit{ is $\E$-essential}  \}.$ Clearly, $\mathfrak{E}$ is an essential collection whose members are fully normalized, which makes  $\mathfrak{E}$ a main essential collection. We obtain that $\E=\langle Aut_{\E}(Q) \mid  Q\in \mathfrak E \ or \ Q=P  \rangle$ by Theorem \ref{thm Alperin 1}.
\end{proof}

Consequently, fully normalized essential subgroups are our main objects right now. Thus, we shall focus on the structure of $Aut_{\E}(Q)$ where $Q$ is a fully normalized $\E$-essentail subgroup.

\begin{lemman}\label{lem aut_p(Q) is strongly closed}
	Let $\E$ be a semisaturated fusion system on a $p$-group $P$ induced by $(\F,S).$ Let $Q$ be a receptive subgroup of $S$ with respect to $\F$. Then $Aut_P(Q)$ is strongly closed in $Aut_S(Q)$ with respect to $Aut_{\F}(Q)$.
\end{lemman}	
	\begin{proof}[\textbf{Proof.}]
		
		Let $c_x\in Aut_P(Q)$ where $x\in N_P(Q)$ and $\psi\in Aut_{\F}(Q)$ such that $c_x^{\psi}\in Aut_S(Q)$. We need to show that $c_x^{\psi}\in Aut_P(Q).$
		
	We see that $c_x^{\psi}\in Aut_S(Q)^{\psi}\cap Aut_S(Q) $, and so $x\in N_{\psi}(Q)$. On the other hand, $Q$ is receptive, and so $\psi$ extends to $\overline \psi\in Hom_{\F}(N_{\psi}(Q),P)$. It then follows that $c_x^\psi=c_x^{\overline \psi}=c_{x\overline \psi}$ on $Q$ by Lemma \ref{iden2}. Since $P$ is strongly $\F$-closed, we obtain that $x\overline \psi \in P$, and so $x\overline \psi\in P\cap N_S(Q)=N_P(Q)$. Thus, $c_x^\psi=c_{x\overline \psi}\in Aut_P(Q)$ as desired.
		
	\end{proof}

\begin{remark}
In the above lemma, $Aut_S(Q)$ is not necessarily a Sylow $p$-subgroup of $Aut_{\F}(Q)$ unless $Q$ is fully $\F$-normalized.
\end{remark}

\begin{lemman}\label{lem fully norm imps strongly closed}
	Let $\E$ be a semisaturated fusion system on a $p$-group $P$ and let $R\leq P$  a fully $\E$-normalized. Let $T\in Syl_p(Aut_{\E}(Q))$ such that $Aut_P(Q)\leq T$. Then $Aut_P(Q)$ is strongly closed in $T$ with respect to $Aut_{\E}(Q)$.
\end{lemman}

\begin{proof}[\textbf{Proof.}] Let $\E$ be induced by $(\F,S)$ where $\F$ is saturated, and let $R\in Q^{\E}=Q^{\F}$ such that $R$ is fully $\F$-normalized.  There exists $\phi \in Hom_{\E}(R,Q)$ such that
	$\phi$ extends to $\overline \phi\in Hom_{\E}(N_P(R),N_P(Q))$ by Corollary  \ref{cor phi extend Hom(N_P(Q), N_P(R)) }.
	
	 As $R$ is a receptive subgroup of $S$ with respect to $\F$, we see that $Aut_P(R)$ is strongly closed in $Aut_S(R)$ with respect to $Aut_{\F}(R)$ by Lemma \ref{lem aut_p(Q) is strongly closed}. Note also that $Aut_{\F}(R)=Aut_{\E}(R)$ as $R\leq P$ and $Aut_S(R)$ is a Sylow $p$-subgroup of $Aut_{\E}(R)$ as $R$ is fully $\F$-normalized. Clearly, $\overline \phi$ induces an isomorphism $\Psi: Aut_{\E}(R) \to Aut_{\E}(Q)$ by $f\mapsto f^{\overline \phi}$ for $f\in Aut_{\E}(R)$. Hence, we get that $(Aut_P(R))\Psi=Aut_P(Q)$ is strongly closed in $T^*=(Aut_S(R))\Psi$ with respect to $Aut_{\E}(Q)$. Notice that $T^*$ is a Sylow $p$-subgroup of $Aut_{\E}(Q)$. It follows that $Aut_P(R)$ is strongly closed in each Sylow subgroup of $Aut_{\E}(Q)$ that contains $Aut_P(R)$, and inparticular in $T$, which proves the claim.

\end{proof}

\begin{definition}\label{def strongly D-embedded}
Let $G$ be a group and $S\in Syl_p(G)$, and $D>1$ be a strongly closed subgroup in $ S$ with respect to $G$. We say that  a proper subgroup $H$ of $G$ is strongly $D$-embedded in $G$ if 
\begin{enumerate}
	\item[i)] $D^x \leq H$ for some $x\in G$.

	\item[ii)]   for all $g\in G-H$, the subgroup $H\cap H^g$ does not contain any $G$-conjugate of any nontrivial subgroup $U$ of $D$.
\end{enumerate}

\end{definition}

\begin{remark}
In the case that $D=S$, being strongly $D$-embedded in $G$ is equivalent to being strongly $p$-embedded in $G$.
\end{remark}

\begin{lemman}\label{lem strongly d-embedded}
	Let $G$ be a group and $S\in Syl_p(G)$, and $D>1$ be a strongly closed subgroup in $ S$ with respect to $G$. Let $H=\langle x\in G \mid D\cap D^x>1 \rangle$. 
	
	\begin{enumerate}
		\item[(a)]  If $H$ is proper in $G$ then $H$ is strongly $D$-embedded in $G$. 
		\item[(b)] Assume that $K$ is strongly $D$-embedded in $G$ and $D\leq K$. Then $H\leq K$. Moreover, $O_p(G)\cap D=1.$ 
	\end{enumerate}

\end{lemman}

\begin{proof}[\textbf{Proof.}] $(a)$
Assume $H<G$.
Clearly we have $S\leq H$ as $D\lhd S$. Pick $x\in G\setminus H$ such that $H\cap H^x$ contains $V>1$ where $V$ is conjugate to a subgroup of $ D$. Then we have $V\leq R\cap T^x$ for some $R,T\in Syl_p(H)$. Since $S\leq H$, set $R=S^a$ and $T=S^b$ for some $a,b\in H$.

 Clearly, $D^a$ is strongly closed in $S^a$ and $V\leq S^a$. Since $V$ is conjugate to a subgroup of $D^a$, we have $V\leq D^a$. Similarly, we have $V\leq D^{bx}$, and so $1<V\leq D^a\cap D^{bx}.$ Then $D\cap  D^{bxa^{-1}}\neq 1$. Thus, we get $bxa^{-1}\in H$, and so $x\in H$, which is a contradiction. Hence $H$ is strongly $D$-embedded in $G$.

$(b)$ Let $K$ be strongly  $D$-embedded in $G$ and $D\leq K$. Let $x\in G$ such that $D\cap D^x>1.$ Then $K\cap K^x$ contains $D\cap D^x>1$, which forces that $x\in K$. Thus, $H\leq K$ as desired. 

Since $S\leq H$, we see that $O_p(G)\leq H\cap H^x$ for any $x\in G- H.$ Note that $H$ is strongly $D$-embedded in $G$ by part $(a)$, and so $ O_p(G)\cap D \leq H\cap H^x\cap D=1.$

\end{proof}

\begin{propositionn}\label{prop charac. of essentails}
Let $\E$ be a semisaturated fusion system on a $p$-group $P$ and $Q$ be a fully normalized $\E$-essential subgroup of $P$. Then the following hold:

\begin{enumerate}
\item[(a)]  $Q$ is $\E$-centric.
\item[(b)] Let $H=\langle \phi \in Out_{\E}(Q) \mid Out_{P}(Q)\cap Out_{P}(Q)^{\phi}>1 \rangle $. Then $H$ is strongly $Out_{P}(Q)$-embedded in $Out_{\E}(Q)$. Moreover, $Out_{P}(Q)\cap O_p(Out_{\E}(Q))=1.$

\item[(c)] $Inn(Q)=Core_G(D)$ where $G=Aut_{\E}(Q)$ and $D=Aut_P(Q)$.
\end{enumerate}
\end{propositionn}

\begin{proof}[\textbf{Proof.}]
Part $(a)$ directly follows from Lemma \ref{lem ess imp f-centric}.

 $(b)$ Set $K=\langle \phi\in Aut_{\E}(Q) \mid N_{\phi}(Q)>Q \rangle$. Since $Q$ is essential, $K$ is proper in $Aut_{\E}(Q)$. Moreover, we see that $Aut_{P}(Q)$ is strongly closed in $S$  with respect to $Aut_{\E}(Q)$ by Lemma \ref{lem fully norm imps strongly closed}   where  $S\in Syl_p(Aut_{\E}(Q))$ such that $Aut_P(Q)\leq S$.
 
 Since $Q$ is $\E$-centric,  $N_{\phi}(Q)>Q$ if and only if $Aut_P(Q)\cap Aut_P(Q)^{\phi}>Inn(Q)$ if and only if $Out_P(Q)\cap Out_P(Q)^{\phi}>1$. It follows that $K/Inn(Q)=H=\langle \phi \in Out_{\E}(Q) \mid Out_{P}(Q)\cap Out_{P}(Q)^{\phi}>1 \rangle $ is proper in $Out_{\E}(Q)$. Note that $Out_P(Q)=Aut_P(Q)/Inn(Q)$ is strongly closed in $S/Inn(Q)$ with respect to $Out_{\E}(Q)=Aut_{\E}(Q)/Inn(Q)$. Then $H$ is strongly $Out_{P}(Q)$-embedded in $Out_{\E}(Q)$ by Lemma \ref {lem strongly d-embedded}(a). We also see that $Out_{P}(Q)\cap O_p(Out_{\E}(Q))=1$ by Lemma \ref {lem strongly d-embedded}(b).
 
 $(c)$ Let $C=Core_G(D)$ where $G=Aut_{\E}(Q)$ and $D=Aut_P(Q)$. We have clearly $Inn(Q)\leq C$ as $Inn(Q)\leq D$. Suppose that $C>Inn(Q)$ and write $\overline C=C/Inn(Q)$. Then $1<\overline C\leq Out_P(Q)$, which yields that $1<\overline C \leq Out_{P}(Q)\cap Out_{P}(Q)^{\phi}$ for all $\phi\in Out_{\E}(Q)$ as $\overline C \unlhd Out_{\E}(Q)$. This contradicts with the fact that $H$ is proper in $Out_{\E}(Q)$. Thus, $C=Inn(Q)$ as desired.
\end{proof}

\begin{corollaryn}\label{cor equivalent def of essential}
	Let $\E$ be a semisaturated fusion system on a $p$-group $P$ and $Q$ be a fully normalized $\E$-centric subgroup of $P$. Then $Q$ is essential if and only if $Out_{\E}(Q)$ has a strongly $Out_{P}(Q)$-embedded subgroup.
\end{corollaryn}

\begin{proof}[\textbf{Proof.}]
	If $Q$ is essential then the result follows by Proposition \ref{prop charac. of essentails}. 	Suppose that $Out_{\E}(Q)$ has a strongly $Out_{P}(Q)$-embedded subgroup. Then $\langle \phi \in Out_{\E}(Q) \mid Out_{P}(Q)\cap Out_{P}(Q)^{\phi}>1 \rangle $ is proper in $Out_{\E}(Q)$ by Lemma \ref{lem strongly d-embedded}. It follows that $\langle \phi \in Aut_{\E}(Q) \mid  N_{\phi}(Q)>QC_P(Q) \rangle$ is proper in $Aut_{\E}(Q)$. Since $Q$ is $\E$-centric, we have $Q=QC_P(Q)$, and so $Q$ is essential. 
	\end{proof}

\begin{lemman} \label{lem N_{E}(Q) is also semisaturated }
	Let $\E$ be a semisaturated fusion system on a $p$-group $P$, and $Q$ be a fully normalized subgroup of $P$. Then the local fusion systems $N_{\E}(Q)$ and $C_{\E}(Q)$ are also semisaturated.
\end{lemman}

\begin{proof}[\textbf{Proof.}]
	Let $\E$ be induced by $(\F,S)$ where $\F$ is saturated, and let $R\in Q^{\E}=Q^{\F}$ such that $R$ is fully $\F$-normalized.  There exists $\phi \in Hom_{\E}(Q,R)$ such that
	$\phi$ extends to an isomorphism $\overline \phi\in Hom_{\E}(N_P(Q),N_P(R))$ by Corollary  \ref{cor phi extend Hom(N_P(Q), N_P(R)) }.
	
	We first note that $N_{\F}(R)$ is saturated as $R$ is fully $\F$-normalized. Notice that $N_P(R)$ is strongly $N_{\F}(R)$-closed. Then it is easy to see that $N_{\E}(R)=N_{\F}(R)_{\mid \leq N_P(R)}$, and so $N_{\E}(R)$ is semisaturated. On the other hand, the isomorphism ${\overline \phi} ^{-1}:N_P(R)\to N_P(Q)$ induces an isomorphism of fusion systems from $N_{\E}(R)$ to $ N_{\E}(Q)$. It then follows that $ N_{\E}(Q)$ is semisaturated as desired. One can similarly obtain that $C_{\E}(Q)$ is semisaturated as well.
\end{proof}

The following lemma creates a bridge between saturated and semisatuated fusion systems.

\begin{lemman}\label{lem correspondence }
	Let $\E$ be a semisaturated fusion system on a $p$-group $P$ induced by $(\F,S)$ and  $Q\leq P$. The following hold:
	\begin{enumerate}
		
		\item[(a)]  $Q$ is strongly  $\E$-closed if and only if $Q$ is strongly $\F$-closed.
		\item[(b)] $Q \unlhd \E$ if and only if $Q\unlhd \F$, and in particular $P \unlhd \E$ if and only if $P\unlhd \F$.
		
	\end{enumerate}
\end{lemman}

	\begin{proof}[\textbf{Proof.}]
For any $x\in Q$, we have $x^{\E}=x^{\F}$ as $Q\leq P$. Thus $(a)$ follows directly.

$(b)$ If $Q\unlhd \F$ then it is obvious that $Q\unlhd \E$. Thus, assume $Q\unlhd \E$ and consider the ascending central series of $Q$;

$$1=Q_0\leq Q_1\leq Q_2\ldots \leq Q_{n-1}\leq Q_n=Q.$$

Since $Q_i$ is characteristic in $Q$, we have $Q_i\unlhd \E$ and in particular, $Q_i$ is strongly $\E$-closed. It follows that $Q_i$ is strongly $\F$-closed by part (a). Then we get that $Q\unlhd \F$ by an Aschbacher's result (see \cite[Proposition 4.62]{Crv}).
	\end{proof}

\begin{corollaryn}\label{cor correspondence}
A $p$-group $P$ is strongly resistant in saturated fusion systems if and only if it is resistant in semisaturated fusion systems.
\end{corollaryn}
%

\begin{theoremn}\label{thm gen  Aschbacher's result}
	Let $\E$ be a semisaturated fusion system on a $p$-group $P$ and $D$ be a strongly $\E$-closed subgroup of $P$. Then $D\unlhd \E$ if and only if $\Omega^*(D)$ has a subgroup series $$1=Q_0\leq Q_1\leq Q_2\ldots \leq Q_{n-1}\leq Q_n=\Omega^*(D)$$ such that 
	
	\begin{enumerate}
		\item[(a)] $Q_i$ is strongly $\E$-closed for $i=0,1,\ldots, n$.
		\item[(b)] $[Q_{i+1},D]\leq Q_i$ for $i=0,1,...,n-1.$
	\end{enumerate}

\end{theoremn}

\begin{proof} [\textbf{Proof.}]

First suppose that $D\unlhd \E$. Set $Q_i=Z_i(D)\cap \Omega^*(D)$ where $Z_i(D)$ is the $i$'th center of $D$. Since $Q_i$ is characteristic in $D$ and $D\unlhd \E$, we see that each $Q_i$ is strongly $\E$-closed. Moreover, $[Q_{i+1},D]\leq Z_i(D)\cap \Omega^*(D)=Q_i$. Thus,

$$1=Q_0\leq Q_1\leq Q_2\ldots \leq Q_{n-1}\leq Q_n=\Omega^*(P)$$ is the desired subgroups series of $\Omega^*(D).$

Now suppose that $\Omega^*(D)$ has such a subgroup series. Let $\E$ be induced by $(F,S)$ and set $\E'=\E_{\mid \leq D}.$ Hence, we see that $\E'=\F_{\mid \leq D}$, that is, $\E'$ is also semisaturated. Moreover,

$$D\unlhd \E' \iff D\unlhd \F \iff D\unlhd \E $$ by Lemma \ref{lem correspondence }(b). Thus, it is enough to show that $D\unlhd \E'$. Note that each $Q_i$ is also strongly $\E'$-closed by Lemma \ref{lem correspondence }(a).

Suppose that $Q$ is a fully normalized $\E'$-essential subgroup of $D$. Clearly, $ \Omega^*(Q)  \leq  \Omega^*(D)$. Set $$S_i= \Omega^*(Q)\cap Q_i$$ for $i=0,1,...,n.$

We get that each $S_i$ are $Aut_{\E}(Q)$-invariant as $\Omega^*(Q)$ is characteristic in $Q$ and each $Q_i$ is strongly $\E'$-closed. Set $A=Aut_D(Q)$. Then we see that

$$[S_{i+1},A]\leq S_i \ \text{for} \ i=0,1,...,n -1 $$ as $[Q_{i+1},D]\leq Q_i$.

 It follows that $Aut_D(Q)\leq O_p(Aut_{\E'}(Q))$ by Lemma \ref{coprime action 3}(b), and so $Out_D(Q)\leq O_p(Out_{\E'}(Q))$. On the other hand, $Out_D(Q)\cap O_p(Out_{\E'}(Q))=1$ by Proposition \ref{prop charac. of essentails}(b), and hence $Out_D(Q)=1$, which is a contradiction. It follows that  $D\unlhd \E'$ by Theorem \ref{thm Alperin 2}, which completes the proof.

\end{proof}
\begin{remark}
For any $p$-group $P$, let $\textbf{W}(P)$ be a characteristic subgroup of $P$, which is uniquely determined by $P$ (more precisely, let  $\textbf{W}$ be a characteristic $p$-functor).	Suppose that for any $p$-group $P$ , the following  hold:  \begin{enumerate}
\item $C_{Aut(P)}(\textbf{W}(P))$ is a $p$-group.	
\item  for all $Q\leq P$, $\textbf{W}(Q)\leq \textbf{W}(P)$.
\end{enumerate}
Then the previous theorem stays correct if we replace $\Omega^*$ with $\textbf{W}$ as we only use the above two properties of $\Omega^*$ in the proof. By taking $\textbf{W}(P)=P$ for all $p$-groups $P$, we have  Aschbacher's result (\cite[Proposition 4.62]{Crv}) for semisaturated fusion systems:
\end{remark}

\begin{corollaryn}\label{cor Aschbacher's lemma}
Let $\E$ be a semisaturated fusion system on a $p$-group $P$, and $D$ be a subgroup of $P$. Then $D\unlhd \E$ if and only if there is a central series  $$1=Q_0\leq Q_1\leq Q_2\ldots \leq Q_{n-1}\leq Q_n=D$$ of $D$ such that $Q_i$ is strongly $\E$-closed for $i=0,1,\ldots, n$.
\end{corollaryn}
%

\begin{lemman}\label{lem E/Q is semisaturated}
	Let $\E$ be a semisaturated fusion system on a $p$-group $P$ and $Q\leq P$ be strongly $\E$-closed. Then $\E/Q$ is a semisaturated fusion system on $P/Q$.
	
\end{lemman}

\begin{proof}[\textbf{Proof.}]
	Let $\E$ be induced by $(\F,S)$ where $\F$ is saturated. Then $Q$ is strongly $\F$-closed  by Lemma \ref{lem correspondence }(a). Hence, $\overline \F=\F/Q$ is saturated by \cite[Proposition 5.11]{Crv}. Clearly, $\overline P=P/Q$ is strongly $\overline \F$-closed.  It follows that $\overline \E=\E/Q$ is semisaturated as $\overline \E=\overline \F_{\mid \leq \overline P}$.
\end{proof}

The following lemma is well known for saturated fusion systems. We generalize this to semisaturated fusion systems. In fact, the same proof also work for quasisaturated fusion systems yet we do not need this.

\begin{lemman}\label{lem U normal E iff in all essential}
	Let $\E$ be a semisaturated fusion system on a $p$-group $P$, and let $U$ be strongly $\E$-closed subgroup of $P$. Then $U\unlhd \E$ if and only if $U$ is contained in all of $\E$-essential subgroups of $P$.
\end{lemman}
\begin{proof}[\textbf{Proof.}]
	The ``if" part is a direct consequences of Theorem \ref{thm Alperin 2}, so let us prove the other direction so assume that $U\unlhd \E$. If there is no $\E$-essential subgroup, there is nothing to prove. Hence, let $Q$ be an $\E$-essential subgroup.
	We first claim that $Aut_U(Q)\unlhd Aut_{\E}(Q)$:
	
	Let $\phi\in Aut_U(Q)$ and $\psi\in Aut_{\E}(Q)$. We need to show that $\phi^{\psi}\in Aut_U(Q)$. Clearly, there exists $x\in U$ such that $\phi=c_x$. Moreover, we observe that $\psi$ extends to $\overline \psi \in Hom_{\E}(QU,P)$ such that $\overline \psi$ leaves invariant $U$ as a consequence of normality of $U$ in $\E$. Then we get:
	$$\phi^{\psi}=c_x^{\psi}=c_{x \overline \psi}$$
	
	by applying Lemma \ref{iden2}. Notice that $x\overline \psi  \in U$, and so $c_{x \overline \psi}=\phi^{\psi}\in Aut_U(Q)$, which proves the claim.
	
	Now recall that $H_Q=\langle \phi \in Aut_{\E}(Q)\mid N_{\phi}(Q) >Q \rangle$ is proper in $Aut_{\E}(Q)$ as $Q$ is essential. Thus, there exists $\phi \in Aut_{\E}(Q)$ such that $N_{\phi}(Q)=Q$. It follows that
	
	$$Aut_{P}(Q)^{\phi}\cap Aut_P(Q)=Inn(Q)$$
	which yields that $Aut_U(Q)\leq Inn(Q)$ as $Aut_U(Q)\unlhd Aut_{\E}(Q)$ and $Aut_U(Q)\leq Aut_P(Q).$ Consequently, we obtain that $U\leq QC_P(Q)=Q$ where the last equality follows from Lemma \ref{lem ess imp f-centric}.

\end{proof}

\begin{corollaryn}\label{cor P normal E iff no essential}
	Let $\E$ be a semisaturated fusion system on a $p$-group $P$.  Then $P\unlhd \E$ if and only if there is no $\E$-essential subgroup of $P$.
\end{corollaryn}

\begin{proof}[\textbf{Proof.}]
	Recall that essentail subgroups must be proper in $P$. The rest directly follows from the previous lemma.
	
	
\end{proof}

\subsection{Proof of Theorems \ref{main thm A}, \ref{A' theroem}  and \ref{main thm C}} In this subsection, we give the proofs of three main theorems and their corollaries.

\begin{proof}[\textbf{Proof of Theorem A}]
Let $\E=\F_{\mid \leq P}$. First we note the following equivalences:

\begin{enumerate}
	\item[$\bullet$] For $Q\leq P$, we have $Q^{\E}=Q^{\F}$.
	\item[$\bullet$] $\F$-centric subgroups with respect to $P$ are exactly $\E$-centric subgroups of $P$.
	
	\item[$\bullet$] Let  $Q$ be fully $\F$-normalized subgroup of $P$. Then  $Q$ is an $\E$-essential subgroup of $P$ if and only if $Q$ is an $\F$-essential subgroup with respect $P$ by Proposition \ref{prop charac. of essentails}, Lemma \ref{lem fully normalized w.r.t F}(c) and Corollary \ref{cor equivalent def of essential}.
\end{enumerate}

Let $\mathfrak U$ be an $\F$-conjugacy representative of proper fully $\F$-normalized subgroups of $P$. By  Corollary \ref{cor fully norm imp reproductive} and Lemma \ref{lem fully normalized w.r.t F}(c), we have $\mathfrak U \subseteq \E^r$. Moreover, $\mathfrak U$ is also $\E$-conjugacy representative of $\E^r$ as each proper subgroup has an $\E$-conjugate lying in $\mathfrak U$. Now set $$\mathfrak{E}=\{Q\in \mathfrak{U}\mid Q \textit{ is $\E$-essential}  \}.$$ Then $\mathfrak E$ is a main collection whose each member is fully $\F$-normalized. Next we see that each member of $\mathfrak E$ is also $\F$-essential subgroup with respect to $P$. We have  $\E=\langle Aut_{\E}(Q) \mid  Q\in \mathfrak E \ or \ Q=P  \rangle$ by Theorem \ref{thm Alperin 2}. Moreover, we also know that $Aut_{\E}(Q)=Aut_{\F}(Q)$, and so Theorem A follows.
\end{proof}

 \begin{proof}[\textbf{Proof of Theorem B}]
Now we are going to show that $P\unlhd \F$ if and only if there is no $\F$-essential subgroup with respect to $P$.

Suppose first that $P\unlhd \F$. Then we get that $P\unlhd \E$ by Lemma \ref{lem correspondence }, and so there is no fully formalized $\E$-essential subgroup by Corollary \ref{cor P normal E iff no essential}. Then we see that there is no $\F$-essential subgroup with respect to $P$, which completes the proof for this part.

Now suppose that there is no $\F$-essential subgroup with respect to $P$. Then it is easy to observe that there exists a main essential collection $\mathfrak{E}$ (as defined in the previous proof) which is empty, and so $P\unlhd \E$ by Theorem \ref{thm Alperin 2}. Then we get $P\unlhd \F$ by Lemma \ref{lem correspondence }, which completes the proof. 
 \end{proof}

We also note that the second part of the above proof can be obtained by applying Theorem A and Lemma \ref{lem correspondence } for this part.

\begin{proof}[\textbf{Proof of Theorem D}]
	The proof of theorem directly follows from Theorem \ref{thm gen  Aschbacher's result} as saturated fusion systems are also semisaturated.
\end{proof}

\begin{proof}[\textbf{Proof of Corollary E}]
	Let $\F$ be a saturated fusion system on $S$ and let $P$ be a strongly $\F$-closed subgroup of $S$.
Let us suppose that $P$ is of odd type and $\Omega(P)\leq Z(P)$. Notice that $\Omega(P)$ is equal to the set $\{x\in P \mid x^p=1\}$ as $\Omega(P)$ is abelian. It follows that $\Omega(P)$ is strongly $\F$-closed as well. Moreover, $[\Omega(P),P]=1$. It then follows that $P\unlhd \F$ by Theorem \ref{main thm C}.
\end{proof}

\begin{proof}[\textbf{Proof of Theorem F}]
	As in our corollary, let us assume that $G$ is a group and $S\in Syl_p(G)$, and $D$ be a strongly closed subgroup in $ S$ with respect to $G$.
	
	We shall prove that $G$ is $p$-nilpotent if and only if $N_G(U)$ is $p$-nilpotent for all $U\leq D$ such that $C_D(U)\leq U$.

Firstly we observe that if $G$ is $p$-nilpotent then all subgroups of $G$ are $p$-nilpotent, and in particular each $N_G(U)$ is $p$-nilpotent, as desired.

Now assume $N_G(U)$ is $p$-nilpotent for all $U\leq D$ such that $C_D(U)\leq U$. Set $\F=\F_S(G)$, which is the fusion system induced by $G$ on $S$. Notice that $D$ is strongly $\F$-closed subgroup of $S$. 

We first claim that there is no $\F$-essential subgroup with respect to $D$. Assume the contrary, let $V$ be an $\F$-essential subgroup with respect to $D$. Then, we see that $C_D(V)\leq V$ as $V$ is $\F$-centric subgroup with respect to $D$, and so $N_G(V)$ is $p$-nilpotent by our hypothesis. It follows that $Aut_{\F}(V)\cong N_G(V)/C_G(V)$ is a $p$-group, which contradicts to the fact that  $Aut_{\F}(V)$ has a strongly $Aut_D(V)$-embedded subgroup. This contradiction proves our claim, and so $D\unlhd \F$, that is, $N_G(D)$ controls $G$-fusion in $S$ by Theorem \ref{A' theroem}. By our hypothesis, $N_G(D)$ is also $p$-nilpotent, and so $G$ is $p$-nilpotent as desired.

\end{proof}

\subsection{Examples}
The following example illustrates the existence of semisaturated fusion system that is not saturated.

\begin{example}
Let $G=S_4$, the symmetric group on four letters, and $S\in Syl_2(G)$. Set $P=V_4$, the klein four group and $\F=\F_S(G)$. Notice that $P\unlhd \F$ as $P\unlhd G$, and in particular $P$ is strongly $\F$-closed. Let $\E=\F_{\mid \leq P}$. Then clearly $\E$ is a semisaturated fusion system on $P$. However, $\E$ is not saturated as $Aut_P(P)\cong 1$ which is not a Sylow $2$-subgroup of $Aut_{\E}(P)=Aut_{\F}(P)\cong S_4/P\cong S_3$.
\end{example}

Now we shall see that there are some groups having strongly $D$-embedded subgroups which do not have a strongly $p$-embedded subgroup.

\begin{example}
	Let $G=S_3 \times C_2$, where $C_2$ denote the cyclic group of order $2$. Set $S=\langle (1,2) \rangle\times C_2 $ and $D=\langle (1,2)\rangle  \times 1 $. Clearly, $S$ is a Sylow $2$-subgroup of $G$ and $D$ is strongly closed in $S$ with respect to $G$. Moreover, $S$ is a strongly $D$-embedded subgroup of $G$ as it can be easily checked. However, $G$ do not have a strongly $2$-embedded subgroup.
	
\end{example}

In the context of direct products, the conclusion of Theorem A can be easily  observed, although  its implications may be considered trivial in this specific instance.
\begin{example}
Let $\F_1$ and $\F_2$ denote saturated fusion systems on $S_1$ and $S_2$, respectively. We can construct a saturated fusion system $\F = \F_1 \times \F_2$ on $S = S_1 \times S_2$. Observe that both $S_1$ and $S_2$ are strongly $\F$-closed subgroups of $S$. In this context, $\F$-essential subgroups with respect to $S_1$ are precisely the $\F_1$-essential subgroups of $S_1$, but none of them are $\F$-essential since they do not contain their centralizers in $S$. However, we can decompose an $\F$-isomorphism in $S_1$, as predicted by Theorem A, by the means of $\F$-essential subgroups with respect to $S_1$.
	
\end{example}

\section{Applications}

In this subsection, we will delve into the application of the theory of semisaturated fusion systems, which we have been developing so far, to specific classes of $p$-groups. Our primary goal is to obtain a proof for Theorem C by combining our understanding of semisaturated fusion systems with some coprime action results.

We will adapt and employ techniques found in \cite{Crv2}, \cite{Drv} and \cite{Stan} to facilitate our investigation. However, due to the less restrictive nature of the automorphism group of essential subgroups in semisaturated fusion systems compared to saturated fusion systems, our proofs will be more intricate and demanding.

To overcome these challenges, we will introduce and explore new techniques related to coprime action and its utility in the study of fusion systems. These coprime action techniques will not only help to simplify our proofs but also lead to novel insights that are valuable in their own right.

We hope that these techniques and the approach through semisaturated fusion systems could lead the discovery of new family of resistant and strongly resistant $p$-groups.
\vspace{0.2cm}

Now we start recalling a definition and proving a lemma, which are frequently addressed throughout the section:\vspace{0.2cm}

A group is called \textbf{$p$-closed} if it has a unique Sylow $p$-subgroup. Note that all subgroups and homomorphic images of a $p$-closed group are $p$-closed. On the other hand, if a group $G$ is not $p$-closed, then $G/N$ is not $p$-closed as well for every normal $p$-subgroup $N$ of $G$.  

\begin{lemman}\label{key lemma}
	Let $\E$ be a semisaturated fusion system on a $p$-group $P$. If $Q$ is a fully normalized essential subgroup of $P$ then  $Aut_{\E}(Q)$ is not $p$-closed.
\end{lemman}
\begin{proof}[\textbf{Proof.}]
	Suppose that  $Aut_{\E}(Q)$ is $p$-closed, and let $S$ be the normal Sylow $p$-subgroup of $Aut_{\E}(Q)$. Since $Aut_P(Q)$ is strongly closed in $S$ with respect to $Aut_{\E}(Q)$ by Lemma \ref{lem fully norm imps strongly closed}, we get that $Aut_P(Q)\unlhd Aut_{\E}(Q)$. Then $Inn(Q)=Core_{Aut_{\E}(Q)}(Aut_P(Q))=Aut_{P}(Q)$ by Proposition \ref{prop charac. of essentails}(c), and so $QC_P(Q)=N_P(Q)$. We get $C_P(Q)\leq Q$  by Proposition \ref{prop charac. of essentails}(a), and so $Q=N_P(Q)$. Hence, $Q=P$ which is a contradiction.

\end{proof}

\subsection{Generalized extraspecial $p$-groups}

Recall that a $p$-group $P$ is said to be \textbf{generalized extraspecial} $p$-group if $\Phi(P)=P'\cong C_p.$ The following result shows that generalized extraspecial $p$-groups are resistant in semisaturated fusion systems with two possible exceptional cases.

\begin{theoremn}\label{thm gen. extraspacial}
	Let $\E$ be a semisaturated fusion system on a generalized extraspecial $p$-group $P$. Then $\E=N_{\E}(P)$ unless $P=E\times A$  for an elementary abelian group $A$ where $E$ is a dihedral group of order $8$ when $p=2$ or is an extraspecial $p$-group of order $p^3$ with exponent $p$ when $p>2$.
\end{theoremn}
We need several preliminary results to prove our theorem and we shall start with a result which is interesting for its own right.

Let $A$ be a group acting on a $p$-group $P$ and let $H$ be an $A$-invariant subgroup of $P$. We \textbf{denote} the set of all $A$-invariant cyclic subgroups of $P$ that are not contained in $H$ by $\mathcal C_{A}(G,H)$.
We give some sufficient conditions for $\mathcal C_{A}(G,H)$ to be nonemty by the means of a counting argument.

\begin{propositionn}\label{prop counting theorem} 
	Let $A$ be a group acting on a $p$-group $P$ via automorphisms of order $r^k$ where  $r$ is a prime dividing $p-1$ and let $H$ be an $A$-invariant subgroup of $P$. Set $|P|=p^n$ and $|H|=p^c$. If $n\not\equiv c \ mod \ r$ then $ \mathcal C_{A}(G,H)$ is nonepmty. Moreover, if $C_G(A)\leq H$, then $A$ acts fixed point freely on each member of $\mathcal C_{A}(G,H)$.
\end{propositionn}

\begin{lemman}\label{counting}
	The number of the nontrivial  cyclic subgroups of a $p$-group of order $p^n$ is congruent to $n$  modulo $p-1$.
\end{lemman}

\begin{proof}[\textbf{Proof}]

	Let $a_k$ be the number of the cyclic subgroups of $P$ of order $p^k$ for $k\geq 0$.	Define an equivalence relation on $P$ by $x \sim y$ if and only if $|x|=|y|$. Then the size of the class of $x$ is equal to $\varphi(p^k) a_{k}$ where $|x|=p^k$ and $\varphi$ is the Euler phi function. Then we get that
	
	$$
	|P|=p^n=\sum\limits_{k=0}^{d} a_k\varphi(p^k)=1+\sum\limits_{k=1}^{d}a_kp^{k-1}(p-1)$$
	
	where $P$ is of exponent $p^d$. It follows that
	$$\dfrac{p^n-1}{p-1}=1+p+p^2+ \ldots +p^{n-1}=\sum\limits_{k=1}^{d}a_kp^{k-1}.$$
	
	Now the left hand side is congurent to $n$ modulo $p-1$ and the right hand side is congruent to $\sum\limits_{k=1}^{d}a_k$ modulo $p-1$. Clearly,  $\sum\limits_{k=1}^{d}a_k$  is equal to the number of the nontrivial  cyclic subgroups of $P$ and the result follows.
\end{proof}

\begin{proof}[\textbf{Proof of Proposition \ref{prop counting theorem}}]
	Let $\mathcal C$ be the nontrivial cyclic subgroups of $P$ that are not contained in $H$. Then  $|\mathcal C|$ of  is congurent to $n-c$  modulo $p-1$ by Lemma \ref{counting}. Since $r$ is a divisor of $p-1$, we get that $|\mathcal C|\equiv n-c \ mod \ r. $ By hypothesis, $ n-c \not\equiv 0 \ mod \ r$, and so $ |\mathcal C| \not\equiv 0 \ mod \ r$.
	
	Now consider the induced action of $A$ on $\mathcal C$. Then we have $|\mathcal C|\equiv |\mathcal C_{A}(G,H)| \ mod \ r$ , and so $\mathcal C_{A}(G,H)\neq \emptyset.$ Now assume $C_G(A)\leq H$ and pick $U\in \mathcal C_{A}(G,H)$. Clearly, $A$ acts nontrivially on $U$ as $U$ is not contained in $H$. Since the action of $A$ on $U$ is coprime, we have $U=[U,A]\times C_U(A)$, and so $C_U(A)=1$ as $U$ is cyclic.
\end{proof}

The following is well known.
\begin{lemman}\label{maximal subgroups}
	Let $p$ be an odd prime and $P$ be a nonabelian metacyclic $p$-group of order $p^3$. Then it has exactly $p$ cyclic maximal subgroups and a unique maximal subgroup isomorphic to $C_p\times C_p$.
\end{lemman}

\begin{lemman}\label{lem generalized extra special}
	Let $P$ be a generalized extraspecial $p$-group. Suppose that one of the maximal subgroup of $P$ is elementary abelian.  Then $P= E\times A$ where $E$ is a nonabelian $p$-group of order $p^3$ and $A$ is elementary abelian. 
\end{lemman}
\begin{proof}[\textbf{Proof.}]
	Let $P$ be a generalized extraspecial $p$-group and $M$ be a maximal subgroup of $P$ such that $M$ is elementary abelian. Let $E$ be a minimal nonabelian subgroup of $P$.
	
	Notice that $P'\leq E$, as otherwise, $P'\cap E=1$ which leads that $E$ is abelian. Clearly $1<\Phi(E)\leq \Phi(P)$, which forces that $\Phi(E)=\Phi(P)\cong C_p.$ Since all maximal subgroups of $E$ are abelian by the minimality of $E$, we have $Z(E)=\Phi(E)$ and $|E:Z(E)|=p^2.$ Thus, $E$ is a nonabelian $p$-group of order $p^3$.
	
	Clearly, $M\cap E$ is a maximal subgroup of $E$. Then we can pick $x\in M\cap E$ and $y\in E- M\cap E$ such that $\langle x,y \rangle=E.$ Since $|P'|=p$, we get that $|P:C_P(y)|=p$, that is, $C_P(y)$ is a maximal subgroup of $P$, and so $|M:C_M(y)|=p$. We see that $C_M(y)\subseteq C_M(x)$ as $x\in M$, which yields that $C_M(y)=C_M(E)$. Write $C=C_M(E)$. Clearly, $P=CE$ and $E\cap C=Z(E)=P'$ as $|P:C|=p^2.$ We see that $C$ is elementary abelian, and so $C=P'\times A$ for some $A\leq C.$ Then $P=E\times A$ as desired.
\end{proof}

\begin{lemman}\label{lem normals are conjuagate}
	Let $\E$ be a semisaturated fusion system on a $p$-group $P$. If $Q$ and $R$ are $\E$-conjugate normal subgroups of $P$, then there exists $\psi \in Aut_{\E}(P)$ such that $R=Q\psi.$
	
\end{lemman}
\begin{proof}[\textbf{Proof.}]
	
	There exists $\psi\in Hom_{\E}(Q,R)$ such that $N_\psi(Q)=N_P(Q)=P$ by Lemma \ref{lem fully normalized w.r.t E}(b). Note that $R$ is fully $\E$-normalized, and so $R$ is receptive by Lemma \ref{lem fully normalized w.r.t E}(a), which yields that $\psi$ extends to $\overline \psi \in Hom_{\E}(N_P(Q),P)=Hom_{\E}(P,P)=Aut_{\E}(P)$.
\end{proof}

\begin{proof}[\textbf{Proof of Theorem \ref{thm gen. extraspacial}}]
	
	Let $\E$ be a semisaturated fusion system on a generalized extraspecial $p$-group $P$ such that $\E\neq N_{\E}(P)$. It then follows that $P'$ is not strongly closed in $P$ by Corollary \ref{cor Aschbacher's lemma}. Hence, there exists $\phi\in Hom_{\E}(P',P)$ such that $R=P'\phi\neq P'.$
	
	Suppose that $R \lhd P$. Then there exists  $\psi \in Aut_{\E}(P)$ such that $R=P'\psi$ by Lemma \ref{lem normals are conjuagate}, which leads a contradiction as $P'\psi=P'.$ It follows that any $Aut_{\E}(P)$-conjugate of $R$ is not normal in $P$.
	
	Since $|P'|=p$ and $R$ is cyclic, we see that $|P:C_P(R)|=p$. Write $M=C_P(R)$. Then $\phi^{-1}$ extends to $\overline {\phi^{-1}} \in Hom_{\E}(M,P)$ by Lemma \ref{lem fully normalized w.r.t E}(a).  Set $N=M\overline {\phi^{-1}}$. Notice that $P'\leq N$ as $N$ is a maximal subgroup of $P$. We see that there exists $\psi \in Aut_{\E}(P)$ such that $N\psi=M$ Lemma \ref{lem normals are conjuagate}, and so $ \overline {\phi^{-1}} \psi\in Aut_{\E}(M)$. Moreover, $R \overline {\phi^{-1}} \psi=P' \psi =P'.$ Since $R\neq P'$, we get that $P'$ is not a characteristic subgroup in $M$. We have that $1\leq \Phi(M)\leq \Phi(P)=P' \cong C_p.$ Then we obtain that $\Phi(M)=1$, and so $M$ is elementary abelian. It follows that $P=E\times A$ where $E$ is a nonabelian $p$-group of order $p^3$ and $A$ is elementary abelian by Lemma \ref{lem generalized extra special}.
	
	Let $p=2$. Then $M\cap E \cong C_2 \times C_2$, and so $E$ is not quaternion, which forces that $E$ is a dihedral group of order $8$ as desired.
	
	It is left to show that $E$ is of exponent $p$ when $p>2.$ Suppose that $E$ is of exponent $p^2$. We see that there exists a fully normalized essential subgroup $Q$ of $P$ by Corollary \ref{cor P normal E iff no essential}. Moreover, $Q=N_i\times A$ for some $i=1,2, \ldots, p+1$ where $N_i$ is a maximal subgroup of $E$ as $C_P(Q)\leq Q$ by Proposition \ref{prop charac. of essentails}(a). Without loss of generality, we set $N_1=C_p\times C_p$ and $N_i=C_{p^2}$ for $i>1$ by Lemma \ref{maximal subgroups}. Notice that $N_1\times A=M$. If $M$ is not essential, then all essentials are of the form $Q=N_i\times A$ for some $i>1$, and hence $\Phi(Q)=P'$. It follows that $P'$ is invariant under $Aut_{\E}(Q)$ and $Aut_{\E}(P)$, which yields that $P'$ is strongly $\E$-closed by Theorem \ref{thm Alperin 2}. This contradiction yields that $M$ is an essential subgroup of $P$.
		
	 Now let's consider $D=Aut_P(M)$. Clearly, $D$ is of order $p$ and centralizes a hyperspace of $M$, that is, $|M:C_M(D)|=p$. Notice that $Aut_{\E}(M)=Out_{\E}(M)$ as $M$ is abelian. We see that $D=Aut_{P}(M)$ is strongly closed in $G=Aut_{\E}(M)$ and $O_p(G)\cap D=1$ by Proposition \ref{prop charac. of essentails}(b). Pick $x\in G$ such that $D\neq D^x$ and set $H=\langle D,D^x \rangle$. We now show that $H$ is not $p$-closed. Since otherwise, $H$ has a unique Sylow $p$-subgroup $S$ containing both $D$ and $D^x$, and so $D=D^x$ by using the fact that $D$ is also strongly closed in $S$ with respect to $H$, which is not the case. Hence, $H$ is not $p$-closed, and in particular $H$ is not a $p$-group.
	 
	  We see that $D^x$ also centralizes a hyperspace of $M$, and so  $p\leq |M:C_M(H)|\leq p^2.$ First suppose that $|M:C_M(H)|=p$. Then both $D$ and $D^x$ acts trivially on $M/C_M(H)$, which yields that $H$ acts trivially on $M/C_M(H)$, and so $[M,H,H]=1$. Note that the action of $H$ on $M$ is faitfull. It follows that $H$ is $p$-group, which is a contradiction. Hence, we get that $|M:C_M(H)|= p^2$ as desired.
	  
	 Now consider the action of $H$ on  $M/C_M(H)\cong C_p\times C_p$. There is an induced homomorphism $f:H\to GL(2,p)$ by the action of $H$ on $M/C_M(H)$. Set $K=Ker(f)$. Let $\phi$ be a $p'$-element of $K$. Then we have $[M,\phi,\phi]=1$, and so $[M,\phi]=1$ by coprime action. Since the action of $K$ on $M$ is faithful, we obtain that $\phi=1$, that is, $K$ is a $p$-group. Since $H$ is not $p$-closed, $f(H)\cong H/K$ is not $p$-closed as well, and so $SL(2,p)\leq f(H)$.
	
	Note that $D\cap K=1$ as $D$ acts nontrivially on $M/C_M(H)$ and $D$ is of order $p$. Let $S$ be a Sylow $p$-subgroup of $H$ such that $D\leq S$. Then $K\lhd S$ as $K$ is a normal $p$-subgroup in $H$. On the other hand, $D\lhd S$ as $D$ is strongly closed in $S$ with respect to $H$. It follows that $[D,K]=1$. Next we obtain that $[D^x,K]=1$ in a similar way, and so $[K,H]=1$ as $H=\langle D,D^x \rangle.$ Consequently, we have $K\leq Z(H).$ 
	
	Let $i$ be the unique involution lying in $SL(2,p)$, and consider $U=f^{-1}(\langle i \rangle)$. Note that $|U:K|=2$ and $K$ is a $p$-group where $p>2$ and $U\lhd H$ as $\langle i \rangle \lhd GL(2,p)$. Let $\langle \varphi \rangle$ be a Sylow $2$-subgroup of $U$. Since $K\leq Z(H)$, we get $\langle \varphi \rangle \lhd U$, and so  $\langle \varphi \rangle$ is characteristic in $U$. It follows that $\langle \varphi \rangle\lhd H$ as $U\lhd H$. Consequently, we get $\varphi\in Z(H) $ as $|\langle \varphi \rangle|=2.$ 
	
	Clearly $\varphi$ centralizes $D=Aut_P(M)$, and so $N_{\varphi}(M)=P$,  and $\varphi$ extends to  some $\alpha \in Aut_{\E}(P)$. Since $p$ is odd, we can also choose $\alpha$ of order $2^k$ for some $k\geq 1$. We observe that $M=\Omega(P)$, and so $\alpha$ is of order $2$ as well. Note also that $$C_M(\alpha)=C_M(\varphi)=C_M(H)< C_M(D)=C_M(P).$$
	
	The second equality comes from the fact that $f(\varphi)=i$ is the unique involution in $SL(2,p)$, which acts fixed point freely on $M/C_M(H)$.
	
	Let $|P|=p^n$. Then $M$ is of order $p^{n-1}$, and clearly $n-(n-1)=1\not\equiv 0 \ mod \ 2$. It follows that there exists an $\alpha$-invariant cyclic subgroup $R$ of $P$ such that $R\nsubseteq M$ by Proposition \ref{prop counting theorem}. Next we have $R\cong C_{p^2}$ as $M=\Omega(P)$. It is easy to observe that $P=MR$ and $M\cap R=P'$ by using the fact that $P=E\times A$. Since $[D,\varphi]=1$ which is the trivial isomorphism of $M$, we get that $[P,\alpha]$ acts trivially on $M$, that is, $[P,\alpha]\leq C_P(M)=M$. Thus, $\alpha$ acts trivially on $P/M$, and so $\alpha$ acts trivially on $R/R\cap M$. Since $R$ is cyclic and $\alpha$ acts coprimely on $R$, we get that $[R,\alpha]=1$, and so $[P',\alpha]=1$. Hence, $P'\leq C_M(\alpha)$.

	Notice that $[M,\alpha]=[P,\alpha]$ as  $[P,\alpha,\alpha]=[P,\alpha] \leq M$. Hence, we get $[M,\alpha] \lhd P$. We also note that $M=[M,\alpha]\times C_M(\alpha)$ by coprime action.
	
	$$[M,\alpha, P]\leq [M,\alpha]\cap P'\leq [M,\alpha]\cap C_M(\alpha)=1.$$
	
	Thus, $P$ centralizes $[M,\alpha]$. $P$ also centralizes $C_M(\alpha)$ as $C_M(\alpha)<C_M(P)$. It then follows that $P$ centralizes  $M$ as $M=[M,\alpha]\times C_M(\alpha)$, which is impossible as $M=C_P(M)$. This final contradiction shows that $E$ is of exponent $p$ when $p>2$ as desired.

\end{proof}

\subsection{Metacyclic $p$-groups} In that subsection, we shall observe that a metacyclic $p$-group is resistant in semisaturated fusion system unless it is dihedral, semidihedral or generalized quaternion.

The following can be obtained by the proof of \cite[Proposition 2.1]{Crv2}. To be precise, we shall give a proof. 
\begin{lemman}\label{lem A act transitly on maximals} Let $P$ be a $2$-generated group. If $A\leq Aut(P)$ is not $p$-closed, then $A$ acts transitively on the set of all maximal subgroups of $P$. In particular, all maximal subgroups of $P$ are isomorphic.
\end{lemman}
\begin{proof}[\textbf{Proof.}]
Consider the action of $A$ on $P/\Phi(P)\cong C_p \times C_p$. Then we have a homomorphism from $A$ to $GL(2,p).$ Since the kernel of this action is a $p$-group, the image of $A$, $\overline A \leq GL(2,p)$ is not $p$-closed, and hence $SL(2,p)\leq  \overline A$. The action of $SL(2,p)$ is transitive on the maximal subgroups of $C_p\times C_p$. Thus, the action of $A$ on  the set of all maximal subgroups of $P$ is transitive as well.
\end{proof}

We need the general version of \cite[Proposition 3.2]{Crv2}(a).

\begin{lemman}\label{lem aut of metacyclic of odd order}
	Let $p$ be an odd prime and $P$ be a metacyclic $p$-group. Then $Aut(P)$ is $p$-closed unless $P\cong C_{p^n}\times C_{p^n}$.
\end{lemman}

\begin{proof}[\textbf{Proof.}]

		Let $P$ be a minimal counter example to the theorem and set $A=Aut(P)$. Clearly, $P$ is not cyclic and  $P$ is a $2$-generated group. Moreover, all maximal subgroups of $P$ are isomorphic groups by Lemma \ref{lem A act transitly on maximals}. It follows that if $P$ is abelian, then $P\cong C_{p^n}\times C_{p^n}$. So we may assume $P$ is nonabelian.
		
		Note that $P'$ is cyclic, and so it has a unique subgroup $U$ of order $p$. Thus, $U$ is a characteristic subgroup of $P$ and $U\leq P'\leq \Phi(P)$. Assume $U<P'$. Then $P/U$ is nonabelian and we obtain that $Aut(P/U)$ is $p$-closed by the minimality of $P$, and so the image of $A$ in $Aut(P/U)$ is $p$-closed. On the other hand, $C_{A}(P/U)$ is a $p$-group, and so $A$ is $p$-closed. This contradiction shows that $P'=U$, that is, $|P'|=p$. In particular, $P'\leq Z(P)$ and $P$ is of class $2$. Since $P$ is of class $2$ and $p>2$, we obtain that the exponent of $\Omega(P)$ is $p$ by \cite[Theorem 4.8]{Isc}.  Note that if $|\Omega(P)|=p$, then $P$ is cyclic by \cite[Theorem 6.11]{Isc}, which is a contradiction. Thus, $|\Omega(P)|\geq p^2$. Clearly $\Omega(P)$ is also metacyclic. Then there exists a cyclic normal subgroup $N$ of $\Omega(P)$ such that $\Omega(P)/N$ is cyclic. It follows that both $N$ and $\Omega(P)/N$ has order $p$ since $\Omega(P)$ is of exponent $p$. Thus, $\Omega(P)\cong C_p\times C_p$.
		
		 Let $\overline A$ be the image of $A$ in $Aut(\Omega(P))$.  Notice that $C_{A}(\Omega(P))$ is a $p$-group, and so $\overline A$ is not $p$-closed. Thus, $A$ acts transitively on the set of maximal subgroups of $\Omega(P)$ by Lemma \ref{lem A act transitly on maximals}. This is not possible as $P'$ is a maximal subgroup of $\Omega(P)$. This contradiction completes the proof.

\end{proof}

\begin{theoremn}\label{thm metacyclic of odd order}
	Let $p$ be an odd prime, $P$ be a metacyclic $p$-group and $\E$ be a semisaturated fusion system on $P$. Then $\E=N_{\E}(P)$.
\end{theoremn}

\begin{proof}[\textbf{Proof.}]
	Let $P$ be a minimal counter example to the theorem and let $\E$ be a minimal semisturated fusion system on $P$ such that $\E\neq N_{\E}(P)$. Thus, we see that $P$ has a fully normalized essential subgroup $Q$ by Theorem \ref{thm Alperin 2}.  Clearly, $P$ is nonabelian. We also note that each proper subgroup of $P$ is also metacyclic.
	
	 We see that $Aut_{\E}(Q)$ is not $p$-closed by Lemma \ref{key lemma}, and so $Aut(Q)$ is not $p$-closed. It follows that $Q\cong C_{p^n}\times C_{p^n}$ by Lemma \ref{lem aut of metacyclic of odd order}. Set $\D= N_{\E}(Q)$. Note that $\D$ is also semisaturated as well by Lemma \ref{lem N_{E}(Q) is also semisaturated }.
	
	Let $R=N_P(Q)$ and assume that $R<P$. The minimality of $P$ forces that $R\unlhd \D$, and so $Aut_{P}(Q)\unlhd Aut_{\E}(Q)$ which is not the case by Lemma \ref{key lemma}. This contradiction shows that $Q\lhd P$. 
	
 Now assume that $\D<\E$. The minimality of $\E$ forces that $N_P(Q)=P\unlhd \D$. Thus, we again have $Aut_P(Q)\unlhd Aut_{\E}(Q)$, which is not possible by Lemma \ref{key lemma}. Hence, we observe that $\D=\E$, that is, $Q\unlhd \E$. 
	
	Now set $U= \Phi (Q)$ and assume that $U\neq 1$. Clearly $U\unlhd \E$ as $U$ is a characteristic subgroup of $Q$. Write $\overline \E=\E/U$ and $\overline P=P/U$. Then $\overline \E$ is a semisaturated fusion system on $\overline P$ by Lemma \ref{lem E/Q is semisaturated}. By the minimality of $P$, we get that $\overline P \unlhd\overline \E$. Thus, we observed that $Aut_{\overline P}(\overline Q)\lhd Aut_{\overline \E}(\overline P)$. Set $A=Aut_{\E}(Q)$ and consider the action of $A$ on $Q/U$. Note that $C_{A}(Q/U)$ is a $p$-group. Set $C=C_{A}(Q/U)$ and $D=Aut_P(Q)$. Then we see that $A/C=Aut_{\overline \E}(\overline Q)$ and $CD/C=Aut_{\overline P}(\overline Q)$, and so $CD\unlhd A$.
We see that $A=Out_{\E}(Q)$ as $Inn(Q)=1$. Thus, there is a strongly $D$-embedded subgroup in $A$ by Proposition \ref{prop charac. of essentails} (b), and moreover $O_p(A)\cap D=1$. Since $CD\leq O_p(A)$, we get that $D=Aut_P(Q)=1$. This is impossible as $Q$ is $\E$-centric and $P>Q$. This contradiction shows that $U=1$ and $Q=C_p\times C_p$. 
	
	Since $Aut_{P}(Q)>1$ and a Sylow $p$-subgroup of $GL(2,p)$ is of order $p$, we see that $Aut_{P}(Q)$ is isomorphic to a Sylow $p$-subgroup of $GL(2,p)$.
	
	Thus, $p=|Aut_{P}(Q)|=|N_P(Q)/C_P(Q)|=|P/Q|$. In particular, we see that $P$ is a nonabelian group of order $p^3$ with exponent $p^2$. Then we get that $P\unlhd \E$ by Theorem \ref{thm gen. extraspacial}, which leads the final contradiction.
\end{proof}

\begin{theoremn}\label{thm 2}
	Let $\E$ be a semisaturated fusion system on a finite metacyclic $2$-group $P$. If $P$ is not dihedral, semidihedral or generalized quaternion, then $\E=N_{\E}(P)$.
\end{theoremn}

We need the following well-known fact.

\begin{lemman}\cite[Sec 5, Theorem 4.5]{Gor} \label{lem |P:P'|=4}
Let $P$ be a nonabelian 2-group such that $|P:P'|=4$. Then $P$ is either dihedral, semidihedral or generalized quaternion group.
\end{lemman}

	\begin{lemman}\cite[Proposition 3.2]{Crv2}\label{lem Aut(P) of metacyclic 2 -groups}
	Let $P$ be a metacyclic $2$-group. If $Aut(P)$ is not a $2$-group then $P\cong Q_8$ or $P\cong C_{2^n}\times C_{2^n}$.
\end{lemman}

We also need to prove \cite[Lemma 2.4]{Crv2} for semisaturated fusion systems.
\begin{lemman} \label{lem |Q|<>4}
Let $\E$ be a semisaturated fusion system on a 2-group $P$ and let $Q$ be an fully normalized abelian $\E$-essential subgroup of $P$. Then the following hold:

\begin{enumerate}
\item [(a)] If $|Q|\leq 4$, then $Q\cong V_4$ and $P$ is dihedral or semidihedral.
\item[(b)] If $|Q|>4$, and $P$ has rank $2$, then $|P:Q|=2$ and $P$ is wreathed.
\end{enumerate}

\end{lemman}

\begin{proof}[\textbf{Proof.}]
$(a)$	We know that $Aut_{\E}(Q)$ is not $2$-closed by Lemma \ref{key lemma}, and so $Q\cong V_4$. Note that $C_P(Q)=Q$ by Proposition \ref{prop charac. of essentails}(a). It follows that $P$ is dihedral or semidihedral by \cite[Corollary 2.3 (3)]{Crv2}.

$(b)$ Since  $Aut_{\E}(Q)$ is not $2$-closed, it is not a $2$-group. We also have that rank of $Q$ is $2$, and so $Q=C_{2^n}\times C_{2^n}$ by \cite[Corollary 2.3 (4)]{Crv2}.

Note that $Aut_{\E}(Q)=Out_{\E}(Q)$ as $Q$ is abelian, and so $O_2(Aut_{\E}(Q))\cap Aut_P(Q)=1$ by Proposition \ref{prop charac. of essentails}(b). Set $A=Aut_{\E}(Q)$. Consider the action of $A$ on $Q/\Phi(Q)\cong V_4$. Note that $C_{A}(Q/\Phi(Q))$ is a $2$-group, and so the image $\overline A$ of $A$ in $GL(2,2)\cong S_3$,  is not $2$-closed, and so $\overline A\cong S_3.$ We get that $C_{A}(Q/\Phi(Q))=O_2(A)$ as $O_2(\overline A)=1$. On the other hand, we have that $Aut_P(Q)\cap O_2(A)=1$ and $Aut_P(Q)>1$ , which forces that $|Aut_P(Q)|=2$. Let $R=N_P(Q)$. Then $|R:Q|=2$. As in the computation carried in  \cite[Lemma 2.4]{Crv2}, we can get that $R$ is  wreathed. Notice that $Q$ is characteristic in $R$. Hence, $Q\lhd N_P(R)\subseteq N_P(Q)=R$. It follows that $R=P$, which completes the proof.
\end{proof}

\begin{proof}[\textbf{Proof of Theorem \ref{thm 2}.}]
	Let $P$ be a minimal counter example to the theorem. We see that $P$ has a fully normalized essential subgroup $Q$ by Theorem \ref{thm Alperin 2}. Clearly, $P$ is nonabelian.
	
	If $Q$ is abelian, then $P$ is dihedral, semidihedral or wreathed by Lemma \ref{lem |Q|<>4}. By our hypotesis, $P$ is not dihedral or semidihedral, and $P$ is not wreathed as $P$ is metacyclic. It follows that $Q\cong Q_8$ by Lemma \ref{lem Aut(P) of metacyclic 2 -groups} and Lemma \ref{key lemma}.
 
	We claim that $Z=Z(P)$ is strongly $\E$-closed.
	Note that $Z\leq C_P(Q)=Z(Q)$ as $Q$ is $\E$-centric, and so $Z=Z(Q)$. Clearly, this is correct for any $\E$-essential subgroup. Since $Z$ is invariant under $Aut_{\E}(Q)$ and $Aut_{\E}(P)$, we get that $Z$ is strongly closed by Theorem \ref{thm Alperin 2}.
	
	Now consider the semisaturated fusion system $\E/Z$ on $P/Z$. We claim that $P/Z\lhd \E/Z$. We can definitely suppose that $P/Z$ is nonabelian.
	
	 Assume that $P/Z$ is dihedral, semidihedral or quaternion. Notice that $P'\geq Q'=Z(Q)=Z(P)$. Thus, $P/P'\cong (P/Z(P))/(P'/Z(P))\cong C_2\times C_2$. This implies that $P$ is either dihedral, semidihedral or quaternion by Lemma \ref{lem |P:P'|=4}, which is not the case. Then by induction applied to $P/Z$, we get that $P/Z \lhd \E/Z$. Write $\overline P=P/Z$. It follows that $Aut_{\overline P}(\overline Q) \lhd Aut_{\overline \E}(\overline Q)$. Note that $|Aut_P(Q)|>4=|Inn(Q)|$, and $Aut_P(Q)\leq Aut(Q) \cong S_4$ which forces that $Aut_P(Q)$ is a Sylow $2$-subgroup of $Aut_{\E}(Q)$. 
	  The kernel of the action of $Aut_{\E}(Q)$ to $Q/Z$ is a $2$-group as $Z=\Phi(P)$, which yields that $Aut_{P}(Q)\lhd Aut_{\E}(Q)$ as $Aut_{\overline P}(\overline Q) \lhd Aut_{\overline \E}(\overline Q).$ We have a contradiction by Lemma \ref{key lemma}, which completes the proof.
 
\end{proof}

\subsection{Some $p$-groups of rank $2$}
The complete classification of $p$-groups of rank $2$ for $p\geq 3$ can be found in \cite{Drv} (see Theorem A.1). In that subsection, we prove the following:

\begin{theoremn}\label{thm p-group of rank 2}
	The following $p$-groups of rank $2$ are resistant in semisaturated fusion systems.
	\begin{enumerate}

		\item[(a)] $P= C(p,r)=\langle a^p=b^p=c^{p^{r-2}}=1 \mid [a,b]=c^{p^{r-3}}, [a,c]=[b,c]=1 \rangle$ where $p\geq 3$  and $r\geq 4$.
		
		\item[(b)] $P= G(p,r,\epsilon)=\langle a^p=b^p=c^{p^{r-2}}=1 \mid [b,c]=1, [a,b^{-1}]=c^{p^{\epsilon r-3}}, [a,c]=b \rangle$ where $p\geq 5$ and $r\geq 4.$
	\end{enumerate}
	
\end{theoremn}

Let $\langle \alpha \rangle$ act on a group $G$ via automorphisms and $n\in \mathbb N$. We say that $\alpha$ acts on $G$ via \textbf{a power automorphism} by $n$ if $g^\alpha =g^n$ for all $g\in G$. 

\begin{lemman}\label{lem if n,m then nm}
	Let $P$ be a nonabelian $p$-group of order $p^3$ and let $\alpha$ be an automorphism of $P$ stabilizing a subgroup series $1<U<M<P$. Suppose that $\alpha$ acts on $P/M$, $M/U$ and $U$ via power automorphism by $n$, $m$ and $k$, respectively. Then the following hold:
	
	\begin{enumerate}
		\item[(a)]	If $U=Z(P)$, then $k\equiv nm \ mod \ p$.
		
		\item[(b)] 	If $U\neq Z(P)$, then $m\equiv nk \ mod \ p$.
	\end{enumerate}

\end{lemman}

\begin{proof}[\textbf{Proof.}]
$(a)$ Suppose that $U=Z(P)$.	Pick $x\in P- M$ and $y \in M-Z(P)$. Then for some $s\in M$ and $t\in Z(P)$ we have $$[x,y]^{\alpha}=[x^{n}s, y^{m}t]=[x^{n}s, y^{m}]=[x^{n}, y^{m}][s, y^{m}]^{x^n}=[x^{n}, y^{m}]=[x,y]^{nm}.$$
	
	Since $[x,y]\neq 1$, we get $Z(P)=\langle [x,y] \rangle$, and the result follows.
	
	$(b)$  Suppose that $U\neq Z(P)$. Then $UZ(P)=M$,  $Z(P)\cong M/U$ and $M/Z(P) \cong U$. It follows that $\alpha$ acts via a power automorphism on $Z(P)$ by $m$ and on $M/Z(P)$ by $k$. Then consider an $\alpha$-invariant subgroup series $1<Z(P)<M<P$ and apply part (a).
\end{proof}

\begin{lemman}\label{lem aut(C(p,r))}
Let $\E$ be a semisaturated fusion system on a $p$-group $P$ where $p\geq 5$, and $Q \cong C(p,r)$ be a fully normalized $\E$-essential subgroup of $P$ where $r\geq 3$. Then there exists $\alpha \in Aut_{\E}(Q)$ of order $q^k$ where $q$ is a prime dividing $p-1$ such that $\langle \alpha \rangle$ normalizes $Aut_{P}(Q)$, $[Aut_P(Q),\alpha]\nleq Inn(Q)$ and $[Z(\Omega(Q)),\alpha]=1$.

\end{lemman}
\begin{proof}[\textbf{Proof.}]
	
	First notice that $Q=\Omega(Q)Z(Q)$ where $\Omega(Q)$ is a nonabelian $p$-group of order $p^3$. $V=\Omega(Q)/\Phi(\Omega (Q))\cong C_p\times C_p.$  We note that $C_{Aut(Q)}(V)$ is a $p$-group, and $Inn(Q)\leq  C_{Aut(Q)}(V)$ as $Q=\Omega(Q)Z(Q)$. Thus, $Out(Q)$ also acts on $V$ as well. 
	
	Since $Q$ is essential, we see that $Aut_{\E}(Q)$ is not $p$-closed by Lemma \ref{key lemma} which leads that $G=Out_{\E}(Q)$ is not $p$-closed as well. Notice that $G$ acts on $V$ induced by the action of $Aut_{\E}(Q)$ on $V$. We have $C_{Aut_{\E}(Q)}(V)$ is a $p$-group, and so $ C_{G}(V)$ is also a $p$-group. Consequently, the image of $G$ in $GL(2,p)$, $\overline G$, is not $p$-closed. Hence, $\overline G\geq SL(2,p)$.
	
	Next we obtain that $C_{G}(V)=O_p(G)$ as $O_p(\overline G)=1$. Set $D=Out_P(Q)$. We have that $D\cap O_p(G)=1$ and $D>1$ is strongly closed in a Sylow $p$-subgroup of $G$ by Proposition \ref{prop charac. of essentails}, and so we see that $|D|=p$, and the image of $D$, $\overline D$, is a Sylow $p$-subgroup of $SL(2,p)$.
	
	Thus, $S=O_p(G)D$ is a Sylow $p$-subgroup of $G$. There exists a cyclic subgroup $\overline U$ of $SL(2,p)$ of order $p-1$ normalizing $\overline D$. It follows that $U$ also normalizes $S$. Note that $D$ is strongly closed in $S$ with respect to $G$. Hence, we get that $U$ also normalizes $D$. Next we observe that $U$ can be chosen of order $p-1$.
	
	Note that $[\overline U, \overline D]\neq 1$, since otherwise we have $ SL(2,p)$ is $p$-nilpotent by Burnside $p$-nilpotency theorem, which is impossible as $p\geq 5$. Thus, $[U,D]\neq 1$, and so there exists $Q\in Syl_q(U)$ such that $[Q,D]\neq 1$. Hence, we get that there exists $\langle \alpha \rangle \leq Aut_{\E}(Q)$ of order $q^k$ such that $\langle \alpha \rangle$ normalizes $Aut_{P}(Q)$ and $[Aut_{P}(Q),\alpha]\nleq Inn(Q)$. 
	
	It is left to show that $\alpha$ acts trivially on $Z(\Omega(Q))$. Since the image of $\alpha$ is in $SL(2,p)$, it is in the form of $\begin{bmatrix} \delta &0\\0&\delta^{-1} \end{bmatrix}$. Note that $ \langle \alpha \rangle $ stabilizes a subgroup series $1<Z(\Omega(Q))<M<\Omega(Q)$. It then follows that $z^{\alpha }=z^{\delta\delta^{-1}}=z$ for all $z\in Z(\Omega(Q))$ by Lemma \ref{lem if n,m then nm}.
\end{proof}

\begin{proof}[\textbf{Proof of Theorem \ref{thm p-group of rank 2}}]

 Let $Q$ be an fully normalized $\E$-essential subgroup of $P$.

Suppose that $(a)$ holds. Since $Q$ is $\E$-centric, we see that $Q\geq C_P(Q) \geq  Z(P)=\langle c \rangle$. On the other hand, $Q>Z(P)$, and so $Q$ is a maximal subgroup of $P$ and $|Q:Z(P)|=p$. It follows that $Q$ is abelian and $Q\cong C_{p^{r-2}}\times C_p$. Since $Aut_{\E}(Q)$ is not $p$-closed by Lemma \ref{key lemma}, we see that $r=3$ by Lemma \ref{lem aut of metacyclic of odd order}. This is impossible as $r\geq 4$. This contradiction shows that there is no fully normalized $\E$-essential subgroup of $P$, and so $P\unlhd \E$ by Theorem \ref{thm Alperin 2}.

Now assume $(b)$ holds. Let $Q$ be a fully normalized essential subgroup of $P$. We see that a proper centric subgroup of $P$ is isomorphic to one of the following by \cite[Lemma A.8]{Drv}; $C_p \times C_{p^{r-2}}$, $C_p \times C_{p^{r-3}}$, $C_{p^{r-2}}$, a nonabelian metacyclic group of order $p^{r-1}$, or $C(p,r-1)$. By using Lemmas \ref{key lemma}, \ref{lem aut of metacyclic of odd order} and the fact that $r\geq 4$, we see that there are two possibilities; either $Q\cong C(p,r-1)$ or, $Q\cong C_p \times C_p$  and $r=4$.

First suppose that $Q\cong C(p,r-1)$. Then $Q=\langle a,b,c^p\rangle $ (see \cite[Lemma A.8]{Drv}). We see that  there exists $\alpha \in Aut_{\E}(Q)$ of order $q^k$ where $q$ is a prime dividing $p-1$  such that $\alpha$ normalizes $Aut_P(Q)$, $[Aut_P(Q),\alpha]\nleq Inn(Q)$ and $[Z(\Omega(Q)),\alpha]=1$ by Lemma \ref{lem aut(C(p,r))}. Firstly, $\alpha$ extends to  $\overline \alpha \in Hom_{\E}(N_P(Q),P) =Aut_{\E}(P)$ as $\alpha$ normalizes $Aut_P(Q)$. Note that we can also choose $\overline \alpha$ of order $q^k$. Since $log_p(|P|/|Q|)=1 \not\equiv \ 0 \mod q $, there exists an $\overline \alpha$-invariant cyclic subgroup $R$ of $P$ such that $R \nsubseteq Q$ by Proposition \ref{prop counting theorem}. Then $R=\langle t\rangle$ where $t=xc^j $ for some $x\in \langle a,b \rangle=\Omega(P)$ and $j$ is a positive number coprime to $p$. Note that  $$t^p=(xc^j)^p=x^pc^{jp}w_1^{p\choose 2}w_2^{p\choose 3}=c^{jp}$$ for some $w_1,w_2 \in \Omega(P)$ and so  $$R\cap Q \geq \langle c^p \rangle \geq Z(\Omega(Q)).$$

Since $\overline \alpha$ acts trivially on $Z(\Omega(Q))$, we obtain that $\overline \alpha$ acts trivially on $R$, which leads that $[P, \overline \alpha]\leq Q$ due to the fact that $P=QR$. Hence, we obtain that $[Aut_P(Q),\alpha]\leq Inn(Q)$. This contradiction shows that $Q\ncong C(p,r-1)$.

Now suppose that $r=4$ and $Q\cong C_p \times C_p$. Then $Q=\langle ab^i, c^p \rangle$ for some $i=0,...,p-1$. Note that $Q$ is a maximal subgroup of $T=\langle a,b,c^p  \rangle$, and in fact $N_P(Q)=T$. Note that $Aut_{\E}(Q)$ is not $p$-closed by Lemma \ref{key lemma}. Then we see that $G=Aut_{\E}(Q)=Out_{\E}(Q)\geq SL(2,p)$, and $D=Out_P(Q)$ has order $p$. Let $\alpha \in N_G(D)\cap SL(2,p)$ of order $q^k$ where $q$ is a prime dividing $p-1$. Note that we can choose $\alpha$ so that $q^k \geq 3$ as $p\geq 5$. We see that $\alpha$ extends to $\psi \in Hom_{\E}(T,P)=Aut_{\E}(T).$ The equality holds as $T=\Omega(P).$ On the other hand, $T\cong C(r,3)$, which is not essential by the previous paragraph. Thus, we obtain the equality
$$H_T=\langle \phi\in Aut_{\E}(T)\mid N_{\phi}(T)>T\rangle= Aut_{\E}(T).$$
We observe that $N_{\phi}(T)=P$ when $N_{\phi}(T)>T$ as $T$ is a maximal subgroup of $P$. The we observe that $\phi$ extends to $\overline \phi\in Aut_{\E}(P)$ as $T$ is fully $\E$-normalized. Then we observe that every automorphism in $Aut_{\E}(T)$ extends to $Aut_{\E}(P)$. In particular, $\psi$ extends to $\overline \psi\in Aut_{\E}(P)$.

 Note that we can also choose $\overline \psi$ of order $q^k$. Since $log_p(|P|/|T|)=1 \not\equiv \ 0 \mod q $, there exists an $\overline \psi$-invariant cyclic subgroup $R=\langle x \rangle $ of $P$ such that $R \nsubseteq T$ by Proposition \ref{prop counting theorem}. Note that $|R|=p^2$ and $R\cap T=Z(T)$, and $RT=P$. We observe that $x^{\overline \psi}=x^n$ for some suitable $n\in \mathbb N.$ Thus, $\overline \psi$ induces a power automorphism on both $Z(T)$ and $P/T$ by $n$. Let $\overline \psi$ induce a power automorphism on $Q/Z(T)$ and $T/Q$ by $i,j$ respectively. 

$$ \overbrace  {1<Z(T}^n \underbrace{ )<Q}_i \underbrace{<}_j \overbrace{T<P}^n.$$

Note that $T$ is a nonabelian group of order $p^3$. Then we get 
$$ij \equiv n \ mod \ p.$$

by Lemma \ref{lem if n,m then nm}(a).

Now consider the subgroup series of $\overline P=P/Z(T)$:  $$\underbrace{\overline 1< \overline Q}_i\underbrace{<}_j \overbrace{\overline T< \overline P}^n .$$

Note that $\overline Q \neq Z(\overline P)$ as $Q$ is not normal in $ P$. It follows that $$j \equiv ni \ mod \ p$$ by Lemma \ref{lem if n,m then nm}(b). Since the restriction of $\overline \psi$ on $Q$, which is $\alpha$, has determinant $1$, $ ni \equiv j\equiv 1 \ mod  \ p$, and so $i \equiv n \ mod \ p$, which yields that $$n^2 \equiv ni \equiv1 \ mod \ p$$

$$i^2 \equiv ni \equiv1 \ mod \ p.$$
 Thus, $\overline \psi ^2$ acts trivially on each section of the subgroup series, and so $\overline \psi^2=1$ due to the coprime action. Since $|\overline \psi|=q^k\geq 3$, we obtain a contradiction, which yields that there is no fully normalized $\E$-essential subgroup of $P$, and so $P\unlhd \E$ by Theorem \ref{thm Alperin 2}.

\subsection{Proof of Theorem C}
We obtain that $P\unlhd \F$ by Theorems \ref{thm gen. extraspacial}, \ref{thm metacyclic of odd order}, \ref{thm 2} \ref{thm p-group of rank 2} and Corollary \ref{cor correspondence}. Now suppose that $(c)$ holds and $P$ is not homocyclic abelian. Then we see that $Aut(P)$ is a $2$-group by Lemma \ref{lem Aut(P) of metacyclic 2 -groups}, and so $\F=SC_{\F}(P)$.
\end{proof}

\section*{Acknowledgements}

I would like to thank Prof. Bob Oliver for his invaluable comments and remarks. The communication with him during his visit at Bilkent University was very helpful and encouraging. \vspace{0.1 cm}

Portions of this article were written at Bilkent University and Middle East Technical University. Throughout that period, I am really grateful to Prof. Ergün Yalçın, Prof. Gülin Ercan, and Prof. Yıldıray Ozan for their assistance and for providing a conducive environment.

Additionally, we acknowledge the financial support provided by TÜBİTAK 2219 grant program, which made this research possible.

\end{document}